\documentclass[12pt]{article}

\usepackage{amsmath, amssymb,amsthm}
\usepackage{ bbold }
\usepackage{color}
\usepackage{hyperref}

\definecolor{darkred}{RGB}{139,0,0}
\definecolor{darkgreen}{RGB}{0,100,0}
\definecolor{darkmagenta}{RGB}{139,0,139}
\definecolor{darkpurple}{RGB}{110,0,180}
\definecolor{darkblue}{RGB}{40,0,200}
\definecolor{darkorange}{RGB}{255,140,0}

\newcommand{\bsa}{\boldsymbol{a}}
\newcommand{\bsb}{\boldsymbol{b}}

\newcommand{\bsgamma}{\boldsymbol{\gamma}}
\newcommand{\bsx}{\boldsymbol{x}}

\newcommand{\bsg}{\boldsymbol{g}}
\newcommand{\bszero}{\boldsymbol{0}}

\newcommand{\bsvarphi}{\boldsymbol{\varphi}}

\newcommand{\bsj}{\boldsymbol{j}}
\newcommand{\bsm}{\boldsymbol{m}}
\newcommand{\bsk}{\boldsymbol{k}}
\newcommand{\bsone}{\boldsymbol{1}}
\newcommand{\bsy}{\boldsymbol{y}}

\newcommand{\bsell}{\boldsymbol{\ell}}
\newcommand{\rd}{\,{\rm d}}
\newcommand{\ld}{{\rm ld}\,}
\newcommand{\RR}{\mathbb{R}}
\newcommand{\NN}{\mathbb{N}}
\newcommand{\ZZ}{\mathbb{Z}}

\newcommand{\FF}{\mathbb{F}}
\newcommand{\cP}{\mathcal{P}}
\newcommand{\cQ}{\mathcal{Q}}

\newcommand{\uu}{\mathfrak{u}}
\newcommand{\vv}{\mathfrak{v}}

\newcommand{\wal}{\mathrm{wal}}

\newtheorem{definition}{Definition}
\newtheorem{theorem}{Theorem}
\newtheorem{corollary}{Corollary}

\newtheorem{remark}{Remark}
\newtheorem{lemma}{Lemma}
\newtheorem{algorithm}{Algorithm}

\allowdisplaybreaks

\begin{document}

\title{Weighted integration over a cube based on digital nets and sequences}

\author{Josef Dick\thanks{J. Dick was supported by the ARC Discovery Project DP190101197 and the Special Research Program “Quasi-Monte Carlo Methods: Theory and Applications” funded by the Austrian Science Fund (FWF) Project F55-N26.}\,\, and Friedrich Pillichshammer\thanks{F.~Pillichshammer
is supported by the Austrian Science Fund (FWF),
Project F5509-N26, which is part of the Special Research Program ``Quasi-Monte Carlo Methods: Theory and Applications''.}}

\maketitle

\begin{abstract}
Quasi-Monte Carlo (QMC) methods are equal weight quadrature rules to approximate integrals over the unit cube with respect to the uniform measure. In this paper we discuss QMC integration with respect to general product measures defined on an arbitrary cube. We only require that the cumulative distribution function is invertible. We develop a worst-case error bound and study the dependence of the error on the number of points and the dimension for digital nets and sequences as well as polynomial lattice point sets, which are mapped to the domain using the inverse cumulative distribution function. We do not require any smoothness properties of the probability density function and the worst-case error does not depend on the particular choice of density function and its smoothness. The component-by-component construction of polynomial lattice rules is based on a criterion which depends only on the size of the cube but is otherwise independent of the product measure.
\end{abstract}

\section{QMC for weighted integration over the cube}\label{sec1}

We are interested in $\varphi$-weighted integration 
\begin{equation}\label{wintunit}
\int_{[\bsa, \bsb]} F(\bsx) \varphi(\bsx)\rd \bsx
\end{equation}
of functions $F$ over the cube $$[\bsa, \bsb] = [a_1, b_1] \times [a_2,b_2]\times \ldots \times [a_s,b_s],$$ where $\bsa = (a_1, \ldots, a_s)$, $ \bsb = (b_1, \ldots, b_s)$, $-\infty < a_i < b_i < \infty$, $\varphi(\bsx)=\prod_{i=1}^s \varphi_i(x_i)$ and, for $i \in [s]:=\{1,2,\ldots,s\}$,  $\varphi_i:[a_i, b_i] \rightarrow \RR$ denotes a probability density function (PDF) with respect to the Lebesgue measure, i.e., $$\forall x \in [a_i,b_i]:\ \varphi_i(x) \ge 0 \ \ \ \mbox{ and }\ \ \ \int_{a_i}^{b_i} \varphi_i(y) \,\mathrm{d} y = 1.$$ The corresponding cumulative distribution functions (CDF) are defined by $$\Phi_i(x) = \int_{a_i}^x \varphi_i(y) \,\mathrm{d} y.  $$ We assume that the $\Phi_i:[a_i,b_i]\rightarrow [0,1]$ are invertible and denote its inverse by $\Phi_i^{-1}:[0,1]\rightarrow [a_i,b_i]$. Let further $\Phi^{-1} = (\Phi^{-1}_1, \ldots, \Phi^{-1}_s)$ be defined as $$\Phi^{-1}(\bsx):=(\Phi^{-1}_1(x_1), \ldots, \Phi^{-1}_s(x_s))\ \ \ \mbox{for $\bsx=(x_1,\ldots,x_s)\in [0,1]^s$.}$$ Note that with $\Phi_i$ also $\Phi_i^{-1}$ is monotone and therefore also Borel-measurable.

Usually, quasi-Monte Carlo (QMC) rules are used for uniform integration over the unit-cube, i.e., when $a_i = 0$, $b_i = 1$ and $\varphi_i \equiv 1$ for all $i \in [s]$. In this context there exists a multitude of literature. For introductory texts and surveys we refer to \cite{DKS13,DP10,LP14,niesiam}. However, often the integration problem is not defined on the unit cube and/or with respect to the uniform measure. The standard approach in this case is to perform a transformation to standardize the problem to a setting where QMC can be applied directly. However, such a transformation can have a big influence on the integration problem and the performance of QMC methods. For instance, \cite{KDSWW08} applied QMC to a problem from statistics, \cite{KSS12} applied QMC to a partial differential equations with random coefficients, \cite{DAFS19} used a tensor approximation to transform QMC points to other distributions appearing in statistics. It is known from \cite{ABN18} that well distributed point sets exist with respect to very general measures, but it is difficult to obtain good explicit constructions (see for instance \cite{NK14} for a successful transformation of lattice point sets to $\mathbb{R}^s$). Here we restrict ourselves to product measures. The more general case of mixture distributions and beyond will be studied in the forthcoming paper \cite{DP2020b}.

\medskip

In this paper we aim at applying QMC for weighted integrals of the form \eqref{wintunit}. We will use QMC rules of the form 
\begin{equation}\label{QMCrule}
\frac{1}{N}\sum_{n=0}^{N-1}F(\Phi^{-1}(\bsy_n)),
\end{equation}
where $\bsy_0,\bsy_1,\ldots,\bsy_{N-1}$ are $N$ points from the $s$-dimensional unit-cube $[0,1]^s$. In particular we study digital nets and sequences as well as polynomial lattice point sets. See \cite{DP10,N87,niesiam} for introductions to these topics. 

\iffalse
One approach to analyse the error of integration in \eqref{QMCrule} would be to consider the function $f = F(\Phi^{-1})$ as the integrand. Since the QMC integration error depends on the smoothness of the integrand, the QMC integration error will in this case also depend on the smoothness of $\Phi^{-1}$. In the following we take a different route by which we avoid this problem.
\fi

Let $\bsgamma:=\{\gamma_\uu \in \RR_+ \ : \ \uu \subseteq [s]\}$ be a given set of positive so-called coordinate weights. These coordinate-weights model in some sense the importance of variables or groups of variables of integrands for the integration problem. This point of view has been introduced by Sloan and Wo\'{z}niakowski~\cite{SW98} in order to explain the effectiveness of QMC also in high dimensions. In fact, in \cite{SW98} the authors used a special type of weights which are so-called {\it product weights} of the form 
 \begin{equation}\label{prodW}
 \gamma_\uu=\prod_{i \in \uu} \gamma_i
 \end{equation}
with a weight sequence $(\gamma_i)_{i \ge 1}$ in $\RR_+$, the set of positive real numbers. 

Another type of weights that attracted a lot of attention during the past years are {\it product order dependend (POD) weights} which are of the form 
\begin{equation}\label{PODw}
\gamma_\uu= \Gamma_{|\uu|} \prod_{j \in \uu} \gamma_j,
\end{equation}
where $(\Gamma_{\ell})_{\ell \ge 0}$ and $(\gamma_j)_{j \ge 1}$ are sequences in $\RR_+$. POD weights are very important in the context of PDEs with random coefficients, see, e.g., \cite{DKLNS14,KSS11}.

\medskip

We consider functions with bounded mixed partial derivatives up to order one in the $\sup$-norm and define, for $p \in [1,\infty]$, the ``weighted'' norm
\begin{equation}\label{def:norm}
\| F \|_{p, s,\bsgamma}:= \left( \sum_{\uu \subseteq [s]} \left( \frac{1}{\gamma_\uu} \sup_{\bsx \in [\bsa, \bsb]} \left | \frac{\partial^{|\uu|} F}{\partial \bsx_\uu} (\bsx)  \right | \right)^p \right)^{1/p},
\end{equation}
with the obvious modifications if $p = \infty$. Here, for $\uu=\{u_1,u_2,\ldots,u_k\}$ with $1 \le u_1 < u_2 < \ldots < u_k \le s$, we write $$\frac{\partial^{|\uu|} F}{\partial \bsx_\uu}:= \frac{\partial^k F}{\partial x_{u_1} \partial x_{u_2} \ldots \partial x_{u_k}}.$$

The integration error for weighted integration using  a QMC-rule \eqref{QMCrule} with underlying set $\cP=\{\bsy_0,\bsy_1,\ldots,\bsy_{N-1}\}$ of integration nodes is defined by $${\rm err}(F; \varphi; \cP):=\frac{1}{N}\sum_{n=0}^{N-1}F(\Phi^{-1}(\bsy_n))- \int_{[\bsa, \bsb]} F(\bsx) \varphi(\bsx)\rd \bsx$$ and the  worst-case error is defined as $${\rm wce}(\cP; p,s,\bsgamma, \varphi):=\sup_{\| F \| _{p, s,\bsgamma}\le 1} |{\rm err}(F; \varphi; \cP)|.$$ In particular, for every $F$ with $\| F \| _{p, s,\bsgamma} < \infty$ we have $$|{\rm err}(F; \varphi; \cP)| \le \| F \| _{p, s,\bsgamma} \ {\rm wce}(\cP; p,s,\bsgamma, \varphi).$$
Notice that the norm of the function does not depend on the integration weight $\varphi$. In this paper we study the decay of the worst-case error ${\rm wce}(\cP; p,s,\bsgamma,\varphi)$ as the number of QMC points $N$ increases. Besides the convergence rate of the worst-case error for $N \to \infty$, we are also interested in the dependence of the integration problem on the dimension $s$. This question is related to tractability studies.

A somewhat surprising result of this paper is that the worst-case error can be bounded independent of $\varphi$ (see Theorem~\ref{thm1} below). The method itself requires that $\Phi^{-1}$ exists and the points $\Phi^{-1}(\bsy_n)$ can be generated, but the decay of the error does otherwise not depend on properties of $\varphi$ or $\Phi^{-1}$. This starkly differs from the straight forward approach were one considers the composition $F \circ \Phi^{-1}$ as the integrand, in which case one ends up with the function norm $\| F \circ \Phi^{-1} \|$ for some norm $\| \cdot \|$. In this case the smoothness properties of $\Phi^{-1}$ are important. Our proof uses Haar functions which are orthogonal with respect to the weight $\varphi$, and we then study the decay of the Haar coefficients with respect to these $\varphi$-orthogonal Haar functions.

\medskip

The paper is organized as follows: In Section~\ref{sec:HaarWalsh} we introduce the basic analytic tools which are Haar- and Walsh functions. The main result of this section is a bound on the Haar coefficients of functions $F$ with bounded norm  $\|F\|_{p, s,\bsgamma}$ (see Lemma~\ref{le:esthcoeff}). In Section~\ref{sec:errest} we present a bound on the integration error which is in many cases easy to handle (see Theorem~\ref{thm1}). This upper bound on the integration error will then be studied for digital nets and sequences as well as for polynomial lattice point sets. The definitions and some basics are recalled in Section~\ref{sec:dignetseqplps}. The error analysis for these node sets follows in Sections~\ref{sec:errdignet}-\ref{sec:PLPS}. The paper concludes with a discussion of the dependence of the worst-case error bounds on the dimension. We will provide sufficient conditions on the weights $\bsgamma$ which guarantee that the obtained error bounds hold uniformly in $s$; the technical term for this property is {\it strong polynomial tractability}; see Section~\ref{sec:SPT}.

\section{Haar and Walsh functions}\label{sec:HaarWalsh}

In the following we recall the definition of Haar functions, which form our basic analytic tool,  estimate the Haar coefficients of functions $F$ with $\|F\|_{p, s,\bsgamma}< \infty$ and show a representation of Haar functions in terms of Walsh functions. 

\paragraph{Haar functions.} Let $j \in \mathbb{N}_{-1} := \{-1, 0, 1, 2, 3, \ldots \}$, $D_{-1} := \{0\}$, and $D_{j} := \{0, 1, \ldots, 2^{j}-1\}$. For $j \in \mathbb{N}_{-1}$ and $m \in D_{j}$ we define the intervals
\begin{equation*}
I_{j, m} := \left[ \frac{m}{2^j}, \frac{m+1}{2^j} \right),
\end{equation*}
and for vectors $\bsj=(j_1,\ldots,j_s)\in \mathbb{N}_{-1}^s$, $\bsm=(m_1,\ldots,m_s)\in D_{\bsj}$, where  $D_{\bsj} := D_{j_1}\times \ldots \times D_{j_s}$, we set
\begin{equation*}
I_{\bsj, \bsm} := \prod_{i=1}^s I_{j_i, m_i}.
\end{equation*}

We define the univariate Haar functions $h_{j,m} :[0,1) \to \mathbb{R}$ by $h_{-1, 0}(x) = 1$, and, for $j>-1$ and $m \in D_j$, by 
\begin{equation*}
h_{j,m}(x) := \begin{cases} 2^{j/2} & \mbox{if } \frac{m}{2^j} \le x < \frac{2m+1}{2^{j+1}}, \\[0.5em] -2^{j/2} & \mbox{if } \frac{2m+1}{2^{j+1}} \le x < \frac{m+1}{2^j}, \\[0.5em] 0 & \mbox{otherwise}.    \end{cases}
\end{equation*}
For $\bsj \in \mathbb{N}_{-1}^s$, $\bsm \in D_{\bsj}$ and $\bsx=(x_1,\ldots,x_s)\in [0,1)^s$ we set $$h_{\bsj, \bsm}(\bsx) := \prod_{i=1}^s h_{j_i, m_i}(x_i).$$ These are the multivariate Haar functions.
These function satisfy the following orthogonality property
\begin{equation}\label{Haar_orthogonal}
\int_{[0,1]^s} h_{\bsj, \bsm}(\bsx) h_{\bsj', \bsm'}(\bsx) \,\mathrm{d} \bsx = \begin{cases} 1 & \mbox{if } (\bsj, \bsm) = (\bsj', \bsm'), \\ 0 & \mbox{otherwise} \end{cases}
\end{equation}

We also define Haar functions defined on $[\bsa, \bsb]$ which are orthogonal with respect to the weight $\varphi$
\begin{equation*}
h^{(\varphi_i)}_{j_i, m_i} (x) := h_{j_i, m_i}(\Phi_i(x)), \quad h^{(\bsvarphi)}_{\bsj, \bsm}(\bsx) := \prod_{i=1}^s h^{(\varphi_i)}_{j_i, m_i}(x_i),
\end{equation*}
where $\bsvarphi:=(\varphi_1,\ldots,\varphi_s)$. Then
\begin{equation*}
\int_{[\bsa, \bsb]} h^{(\varphi)}_{\bsj, \bsm}(\bsx) h^{(\varphi)}_{\bsj', \bsm'}(\bsx) \varphi(\bsx) \,\mathrm{d} \bsx = \begin{cases} 1 & \mbox{if } (\bsj, \bsm) = (\bsj', \bsm'), \\ 0 & \mbox{otherwise}, \end{cases}
\end{equation*}
which is equivalent to \eqref{Haar_orthogonal}.

Let $F: [\bsa, \bsb] \to \mathbb{R}$ with $\| F \| _{p,s,\bsgamma}< \infty$. We expand $F$ in its Haar series
\begin{equation}\label{f:HaarSeries}
F(\bsx) = \sum_{\bsj \in \mathbb{N}_{-1}^s} \sum_{\bsm \in D_{\bsj}} \widehat{F}^{(\bsvarphi)}(\bsj,\bsm) h^{(\bsvarphi)}_{\bsj, \bsm}(\bsx),
\end{equation}
where, due to the $L_2$ orthonormality of $h^{(\bsvarphi)}_{\bsj, \bsm}$, we have
\begin{equation*}
\widehat{F}^{(\bsvarphi)}(\bsj, \bsm) = \int_{[\bsa, \bsb]} F(\bsx) h^{(\bsvarphi)}_{\bsj, \bsm}(\bsx) \varphi(\bsx) \,\mathrm{d} \bsx.
\end{equation*}

In order to estimate the integration error we need to know the decay rate of $\widehat{F}^{(\bsvarphi)}(\bsj, \bsm)$. For a proper statement of this estimate we use the following notation: For $\uu \subseteq [s]$ and $\bsj\in \NN_{-1}^s$ set $\bsj_\uu := (j_i)_{i \in \uu}$ and let $(\bsj_\uu, - \bsone) \in \mathbb{N}_{-1}^s$ be the vector whose $i^{{\rm th}}$ component is $j_i$ if $i \in \uu$ and $-1$ otherwise. We use a similar notation for $(\bsm_\uu, \bszero)$.

\begin{lemma}\label{le:esthcoeff}
Let $F:[\bsa, \bsb] \rightarrow \RR$ with $\|F\|_{p, s,\bsgamma}< \infty$. For $\emptyset \not= \uu \subseteq [s]$, $\bsj_\uu \in \NN_0^{|\uu|}$ and $(\bsm_\uu,\bszero) \in D_{(\bsj_\uu,-\bsone)}$ we have $$|\widehat{F}^{(\bsvarphi)}((\bsj_\uu, -\bsone), (\bsm_\uu, \bszero)) | \le  \frac{1}{2^{ |\bsj_\uu| / 2 }} \left(\prod_{i \in \uu} \lambda_{j_i, m_i} \right) \sup_{\bsx \in [\bsa, \bsb]} \left| \frac{\partial^{|\uu|} F}{\partial \bsx_\uu} (\bsx) \right|, $$ where
\begin{equation}\label{deflambda}
\lambda_{j_i,m_i} = \frac{1}{2} \left[ \Phi^{-1}_i\left( \frac{m_i + 1}{2^{j_i}}\right) -  \Phi^{-1}_i\left(\frac{m_i}{2^{j_i}}\right) \right].
\end{equation}
\end{lemma}

\begin{proof}
To simplify the notation we write $a$, $b$, $\varphi$, $\Phi$ and $\Phi^{-1}$ instead of $a_i$, $b_i$, $\varphi_i$, $\Phi_i$ and $\Phi_i^{-1}$ here. Define $H_{j,m}:[a,b] \to \RR$ by
\begin{align*}
H_{j, m}(x) := & \int_a^x h^{(\varphi)}_{j, m}(y) \varphi(y) \,\mathrm{d} y = \int_0^{\Phi(x)}  h_{j, m}(y) \,\mathrm{d} y \\ =  &  2^{j/2} \times \begin{cases}  \Phi(x) - \frac{m}{2^j} & \mbox{if } \Phi^{-1}(\frac{m}{2^j}) \le x < \Phi^{-1}(\frac{m+1/2}{2^j}),  \\[0.5em] \frac{m+1}{2^j} - \Phi(x) & \mbox{if } \Phi^{-1}(\frac{m+1/2}{2^j}) \le x < \Phi^{-1}(\frac{m+1}{2^j}), \\[0.5em] 0 & \mbox{otherwise}.  \end{cases}
\end{align*}
From this we directly obtain for $(j,m)\not=(-1,0)$ that $$H_{j, m}(a) = H_{j, m}(b) = H_{j,m}\left(\Phi^{-1}\left(\frac{m}{2^j}\right)\right) = H_{j,m}\left(\Phi^{-1}\left(\frac{m+1}{2^j}\right)\right) = 0,$$ $$H_{j,m}\left(\Phi^{-1}\left(\frac{2m+1}{2^{j+1}}\right)\right) = \frac{1}{2^{1+j/2}},$$ $$0 \le H_{j, m}(x) \le \frac{1}{2^{1+j/2}}\quad \mbox{ for all $x \in [a, b]$,}$$ and further
\begin{equation*}
\int_a^b H_{j, m}(x) \,\mathrm{d} x \le \frac{1}{2^{1+j/2}} \left[ \Phi^{-1}\left(\frac{m+1}{2^j}\right) - \Phi^{-1}\left(\frac{m}{2^j}\right)  \right].
\end{equation*}

Using integration by parts in each coordinate $i \in \uu$ we obtain
\begin{align*}
\widehat{F}^{(\bsvarphi)}((\bsj_\uu, -\bsone), (\bsm_\uu, \bszero)) = &  \int_{[\bsa, \bsb]} F(\bsx) h^{(\bsvarphi)}_{(\bsj_\uu, -\bsone), (\bsm_\uu, \bszero) }(\bsx) \varphi(\bsx) \,\mathrm{d}\bsx \\  =  & (-1)^{|\uu|}  \int_{[\bsa, \bsb]} \frac{\partial^{|\uu|} F}{\partial \bsx_\uu} (\bsx) \prod_{i \in \uu} H_{j_i, m_i}(x_i) \prod_{i \notin \uu} \varphi_i(x_i) \,\mathrm{d} \bsx.
\end{align*}
This implies that
\begin{align*}
|\widehat{F}^{(\bsvarphi)}((\bsj_\uu, -\bsone), (\bsm_\uu, \bszero)) | \le &  \sup_{\bsx \in [\bsa, \bsb]} \left| \frac{\partial^{|\uu|} F}{\partial \bsx_\uu} (\bsx) \right| \,  \prod_{i \in \uu} \int_{a_i}^{b_i} H_{j_i, m_i}(x) \,\mathrm{d} x \\ \le &  \frac{1}{2^{ |\bsj_\uu| / 2 }} \left(\prod_{i \in \uu} \lambda_{j_i, m_i} \right) \sup_{\bsx \in [\bsa, \bsb]} \left| \frac{\partial^{|\uu|} F}{\partial \bsx_\uu} (\bsx) \right|,
\end{align*}
where $\lambda_{j_i, m_i}$ is as in \eqref{deflambda}.
\end{proof}

The bound on the Haar coefficients $\widehat{F}^{(\bsvarphi)}(\bsj, \bsm)$ in Lemma~\ref{le:esthcoeff} depends on $\Phi^{-1}$. When estimating the worst-case error later on we can remove this dependence since there we only need to estimate sums of the $\lambda_{j_i,m_i}$ and then it is enough to use the property that $\varphi_i$ is a probability density function. More precisely, we have:
\begin{lemma}\label{S1}
We have
$$\sum_{\bsm_\uu \in D_{\bsj_\uu}} \prod_{i \in \uu} \lambda_{j_i, m_i}=\prod_{i \in \uu} \frac{b_i - a_i}{2}.$$ 
\end{lemma}

\begin{proof}
We have
\begin{equation*}
\sum_{\bsm_\uu \in D_{\bsj_\uu}} \prod_{i \in \uu} \lambda_{j_i, m_i} = \prod_{i \in \uu} \sum_{k=0}^{2^{j_i}-1} \lambda_{j_i, k} = \prod_{i \in \uu} \frac{1}{2} \underbrace{\left( \Phi^{-1}_i(1) - \Phi^{-1}_i(0) \right)}_{ = b_i - a_i} =  \prod_{i \in \uu} \frac{b_i - a_i}{2}.
\end{equation*}
\end{proof}

Note that in Lemma~\ref{S1} we need the assumption that $-\infty < a_i < b_i < \infty$.

\paragraph{Walsh functions.} Later, when we analyse the integration error, it will be convenient to represent the Haar functions in terms of Walsh functions. In the following we introduce Walsh functions and show a well known connection to Haar functions. %This is due to the group structure of polynomial lattices whereby the Walsh functions form the corresponding dual group. 

Let the real number $x \in [0,1)$ have dyadic expansion $x =
\frac{\xi_1}{2} + \frac{\xi_2}{2^2} + \cdots$, with digits $\xi_i \in \{0,1\}$. For any dyadic expansion we assume that infinitely many digits are different from $1$, which makes the expansion unique. For $k \in
\mathbb{N}_0$ with dyadic expansion $k = \kappa_0+\kappa_1 2 +\cdots+\kappa_{r-1} 2^{r-1}$ and $
\kappa_0,\ldots, \kappa_{r-1} \in \{0,1\}$, define the $k^{{\rm th}}$ (dyadic) Walsh function
by
$$\mathrm{wal}_k(x) := (-1)^{\kappa_0 \xi_1+\kappa_1 \xi_2+\cdots + \kappa_{r-1} \xi_r}.$$ 

In dimensions $s > 1$ we use products of the Walsh functions. Let $\bsx = (x_1, x_2,\ldots, x_s) \in [0,1]^s$ and $\bsk = (k_1, k_2, \ldots, k_s) \in \mathbb{N}_0^s$. Then we define the $\bsk^{{\rm th}}$ (dyadic) Walsh function by
\begin{equation*}
\mathrm{wal}_{\bsk}(\bsx) := \prod_{i=1}^s \wal_{k_i}(x_i).
\end{equation*}

We have the following well known representation of Haar functions in terms of Walsh functions.
\begin{lemma}\label{le1}
For $j \in \NN_0$, $m\in D_j$ and $x\in [0,1)$ we have $$h_{j,m}(x)=\frac{1}{2^{j/2}} \sum_{k=0}^{2^j-1} \wal_{k+2^j}\left(x \ominus \frac{m}{2^j}\right).$$ Here $\ominus$ denotes the digit-wise dyadic subtraction, i.e., for $x=\frac{\xi_1}{2}+\frac{\xi_2}{2^2}+\cdots$ and $y=\frac{\eta_1}{2}+\frac{\eta_2}{2^2}+\cdots$ with digits $\xi_i,\eta_i \in \{0,1\}$ we set $x \ominus y:= \frac{\zeta_1}{2}+\frac{\zeta_2}{2^2}+\cdots$ with $\zeta_i:=\xi_i-\eta_i \pmod{2}$.
\end{lemma}

\begin{proof}
For completeness we provide a short proof. Let $m=m_0+m_1 2+ \cdots + m_{j-1} 2^{j-1}$ and $x=\frac{\xi_1}{2}+\frac{\xi_2}{2^2}+\cdots$ be the binary expansions of $m$ and $x$, respectively. Then we have
\begin{align*}
\frac{1}{2^{j/2}} \sum_{k=0}^{2^j-1} \wal_{k+2^j}\left(x \ominus \frac{m}{2^j}\right) = &  \frac{1}{2^{j/2}}  \sum_{\kappa_0,\ldots,\kappa_{j-1}=0}^1 (-1)^{\kappa_0(\xi_1 \ominus m_{j-1})+\cdots + \kappa_{j-1}(\xi_j \ominus m_0)+ \xi_{j+1}} \\
= &  \frac{1}{2^{j/2}} (-1)^{\xi_{j+1}} \prod_{r=1}^j \sum_{\kappa=0}^1 (-1)^{\kappa (\xi_r \ominus m_{j-r})}\\
= & \left\{ 
\begin{array}{ll}
0 & \mbox{ if } x \not \in I_{j,m},\\
2^{j/2} (-1)^{\xi_{j+1}} & \mbox{ if } x \in I_{j,m},
\end{array}\right.\\
= & h_{j,m}(x).
\end{align*}
\end{proof}

\section{General error analysis}\label{sec:errest}

In this section we present a general error analysis for functions $F$ with bounded norm $\| F \|_{p,s,\bsgamma} $. We obtain an upper bound on the integration error that can be used later on for digital nets and sequences as well as for polynomial lattice point sets. The error analysis requires the use of projection regular point sets with $N=2^m$ elements in $[0,1)^s$.

\begin{definition}[Projection regular point set]\rm 
Let $m \in \NN$ and let $\cP=\{\bsx_0,\bsx_1,\ldots,\bsx_{2^m-1}\}$ be a $2^m$-element point set in $[0,1)^s$. Let $\bsx_n=(x_{n,1},\ldots,x_{n,s})$ for $n \in \{0,\ldots, 2^m -1\}$ and let $\cP_{\{i\}}=\{x_{0,i},x_{1,i},\ldots,x_{2^m-1,i}\}$ for $i \in [s]$ be the one-dimensional projections of $\cP$. We call $\cP$ {\it projection regular}, if for every $i \in [s]$ every interval $[\frac{a}{2^m},\frac{a+1}{2^m})$, $a \in D_m$, contains exactly one element from $\cP_{\{i\}}$.  
\end{definition}

The error bound involves a function $\mu$ on $\NN_0$ which measures the length of the binary digit expansion of a natural number. 

\begin{definition}\rm \label{def:mu}
For $k \in \NN_0$ we define $\mu(k)=0$ if $k=0$ and $\mu(k)=l$ for some $l \in \NN$ if $k$ has binary expansion $k=\kappa_0+\kappa_1 2+\cdots + \kappa_{l-1} 2^{l-1}$ with $\kappa_j \in \{0, 1\}$ for all $j \in \{0,1,\ldots,l-1\}$ and $\kappa_{l-1}=1$. 

For $\bsk=(k_1,k_2,\ldots,k_s)\in \NN_0^s$ we define $$\mu(\bsk)=\mu(k_1)+\mu(k_2)+\cdots +\mu(k_s).$$
\end{definition}

Now we can state the announced error estimate.

\begin{theorem}\label{thm1}
Let $m \in \NN$ and let $\cP= \{\bsy_0,\ldots,\bsy_{2^m-1}\}$ be a projection regular $2^m$-element point set in $[0,1)^s$. Let $1 \le p,q \le \infty$ be such that $\frac{1}{p}+\frac{1}{q}=1$. Then we have
\begin{equation*}
{\rm wce}(\cP; p,s,\bsgamma, \varphi)\le \left(\sum_{\emptyset \neq \uu \subseteq [s]}  \gamma_\uu^q \prod_{i \in \uu} (b_i - a_i)^q \left(\frac{2^{|\uu|}}{2^m}  + E(\cP,\uu)   \right)^q \right)^{1/q},
\end{equation*}
with the obvious modifications if $q = \infty$, and where 
\begin{equation}\label{def:epu}
E(\cP,\uu):=\sum_{\bsk_\uu \in \{1, \ldots, 2^m-1\}^{|\uu|}} \frac{1}{2^{\mu(\bsk_\uu)}} \left|\frac{1}{2^m} \sum_{n=0}^{2^m -1} \mathrm{wal}_{\bsk_\uu}(\bsy_{n,\uu})\right|.
\end{equation}
Here, for $\uu \subseteq [s]$, $\bsy_{n,\uu}$ denotes the projection of the point $\bsy_n$ to the components which belong to $\uu$ and likewise $\bsk_\uu=(k_i)_{i \in \uu}$.
\iffalse
\begin{equation}\label{defbgqp}
B_{\bsgamma,q}(\cP) :=
\left( \sum_{\emptyset \neq \uu \subseteq [s]} \gamma_\uu^q \prod_{i \in \uu} (b_i-a_i)^q  \left( \sum_{\bsk_\uu \in \{1, \ldots, 2^m-1\}^{|\uu|}} \frac{1}{2^{\mu(\bsk_\uu)}} \left|\frac{1}{2^m} \sum_{n=0}^{2^m-1} \mathrm{wal}_{\bsk_\uu}(\bsy_{n,\uu})\right|  \right)^q  \right)^{1/q}.
\end{equation}
\fi
\end{theorem}

\begin{proof}
Conveniently we have
\begin{equation*}
h^{(\varphi)}_{\bsj, \bsm}(\Phi^{-1}(\bsy_n)) = h_{\bsj, \bsm}(\bsy_n).
\end{equation*}
Using this property and the Haar series expansion \eqref{f:HaarSeries} the integration error can be rewritten as
\begin{align}\label{est:err1}
\mathrm{err}(F; \varphi; \cP) & = \frac{1}{2^m} \sum_{n=0}^{2^m-1} F(\Phi^{-1}(\bsy_n)) - \int_{[\bsa, \bsb]} F(\bsx) \varphi(\bsx) \,\mathrm{d} \bsx \nonumber\\
& =  \sum_{\stackrel{\bsj \in \mathbb{N}_{-1}^s}{\bsj \neq -\bsone}} \sum_{\stackrel{ \bsm \in D_{\bsj} }{\cP \cap I_{\bsj, \bsm} \neq \emptyset} } \widehat{F}^{(\bsvarphi)}(\bsj, \bsm) \frac{1}{2^m} \sum_{n=0}^{2^m-1} h_{\bsj, \bsm}(\bsy_n).
\end{align}

For $\uu \subseteq [s]$ let $\cP_\uu$ be the set of projections of the points from $\cP$ to the components which belong to $\uu$, i.e., $$\cP_\uu=\{\bsy_{0,\uu}, \bsy_{1,\uu},\ldots ,\bsy_{2^m -1,\uu}\}.$$ Now we rearrange the sum over all $\bsj \in \NN_{-1}^s\setminus \{-\bsone\}$ in \eqref{est:err1} according to the sets of components of $\bsj$'s which differ from $-1$. We have 
\begin{align}\label{err_bound1}
\lefteqn{| \mathrm{err}(F;\varphi;\cP)|}\nonumber \\
& =  \left| \sum_{\emptyset \neq \uu \subseteq [s]} \sum_{\bsj_\uu \in \mathbb{N}_0^{|\uu|}} \sum_{\stackrel{ \bsm_\uu \in D_{\bsj_\uu} }{\cP_\uu \cap I_{\bsj_\uu, \bsm_\uu} \neq \emptyset}} \widehat{F}^{(\bsvarphi)}((\bsj_\uu, - \bsone), (\bsm_\uu, \bszero)) \frac{1}{2^m} \sum_{n=0}^{2^m-1} h_{(\bsj_\uu, -\bsone), (\bsm_\uu, \bszero)}(\bsy_n) \right| \nonumber \\ 
& \le   \sum_{\emptyset \neq \uu \subseteq [s]} \sum_{\bsj_\uu \in \mathbb{N}_0^{|\uu|}} \sum_{\stackrel{ \bsm_\uu \in D_{\bsj_\uu} }{\cP_\uu \cap I_{\bsj_\uu, \bsm_\uu} \neq \emptyset}}  \left| \widehat{F}^{(\bsvarphi)}((\bsj_\uu, - \bsone), (\bsm_\uu, \bszero)) \right| \left| \frac{1}{2^m} \sum_{n=0}^{2^m-1} h_{(\bsj_\uu, -\bsone), (\bsm_\uu, \bszero)}(\bsy_n) \right| \nonumber \\ 
& \le \| F\|_{p, s,\bsgamma} \left( \sum_{\emptyset \neq \uu \subseteq [s]} \left( \gamma_\uu \sum_{\bsj_\uu \in \mathbb{N}_0^{|\uu|}} \frac{1}{2^{|\bsj_\uu|/ 2}}  \sum_{\stackrel{ \bsm_\uu \in D_{\bsj_\uu} }{\cP_\uu \cap I_{\bsj_\uu, \bsm_\uu} \neq \emptyset}}  \prod_{i \in \uu} \lambda_{j_i, m_i} \left| \frac{1}{2^m} \sum_{n=0}^{2^m-1} h_{(\bsj_\uu, -\bsone), (\bsm_\uu, \bszero)}(\bsy_n) \right| \right)^q \right)^{1/q},   
\end{align}
where we used Lemma~\ref{le:esthcoeff} and H\"older's inequality, which is justified because  $\frac{1}{p}+\frac{1}{q}=1$.

Denote the summand, that is raised to the power of $q$ by $S_\uu$, i.e.,
\begin{equation}\label{defSu}
S_\uu:=\gamma_\uu \sum_{\bsj_\uu \in \mathbb{N}_0^{|\uu|}} \frac{1}{2^{|\bsj_\uu|/ 2}}  \sum_{\stackrel{ \bsm_\uu \in D_{\bsj_\uu} }{\cP_\uu \cap I_{\bsj_\uu, \bsm_\uu} \neq \emptyset}}  \prod_{i \in \uu} \lambda_{j_i, m_i} \left| \frac{1}{2^m} \sum_{n=0}^{2^m-1} h_{(\bsj_\uu, -\bsone), (\bsm_\uu, \bszero)}(\bsy_n) \right|.
\end{equation}
In order to estimate $S_\uu$ we distinguish two cases. We write $$S_\uu=S_\uu^{(1)}+S_\uu^{(2)},$$ where 
$$S_\uu^{(1)}:=\gamma_\uu \sum_{\bsj_\uu \in \{0,1,\ldots,m-1\}^{|\uu|}} \frac{1}{2^{|\bsj_\uu|/ 2}}  \sum_{\stackrel{ \bsm_\uu \in D_{\bsj_\uu} }{\cP_\uu \cap I_{\bsj_\uu, \bsm_\uu} \neq \emptyset}}  \prod_{i \in \uu} \lambda_{j_i, m_i} \left| \frac{1}{2^m} \sum_{n=0}^{2^m-1} h_{(\bsj_\uu, -\bsone), (\bsm_\uu, \bszero)}(\bsy_n) \right|$$ and $$S_\uu^{(2)}:=S_\uu - S_\uu^{(1)}.$$

We consider $S_\uu^{(1)}$ where the summation is restricted to all $\bsj_\uu \in \{0, 1, \ldots, m-1\}^{|\uu|}$. We separate the sum over $\bsm_\uu$ into two parts using
\begin{eqnarray*}
\lefteqn{\sum_{\stackrel{ \bsm_\uu \in D_{\bsj_\uu} }{\cP_\uu \cap I_{\bsj_\uu, \bsm_\uu} \neq \emptyset}}  \prod_{i \in \uu} \lambda_{j_i, m_i} \left| \frac{1}{2^m} \sum_{n=0}^{2^m-1} h_{(\bsj_\uu, -\bsone), (\bsm_\uu, \bszero)}(\bsy_n) \right|}\\ 
& \le & \left(\sum_{\bsm_\uu \in D_{\bsj_\uu}}  \prod_{i \in \uu} \lambda_{j_i, m_i}\right) \max_{\bsell_\uu \in D_{\bsj_\uu}}  \left| \frac{1}{2^m} \sum_{n=0}^{2^m-1} h_{(\bsj_\uu, -\bsone), (\bsell_\uu, \bszero)}(\bsy_n) \right| \\ 
& \le &  \prod_{i \in \uu} \frac{b_i - a_i}{2} \max_{\bsell_\uu \in D_{\bsj_\uu}}  \left| \frac{1}{2^m} \sum_{n=0}^{2^m-1} h_{(\bsj_\uu, -\bsone), (\bsell_\uu, \bszero)}(\bsy_n) \right|,
\end{eqnarray*}
where we used Lemma~\ref{S1} in the last step. 

Further, according to Lemma~\ref{le1},
\begin{equation*}
h_{(\bsj_\uu, - \bsone), (\bsm_\uu, \bszero)}(\bsx) = \frac{1}{2^{|\bsj_\uu|/2}} \sum_{\bsk_\uu \in D_{\bsj_\uu}} \mathrm{wal}_{\bsk_\uu + 2^{\bsj_\uu}} \left(\bsx_\uu \ominus \frac{\bsm_\uu}{2^{\bsj_\uu}} \right),
\end{equation*}
where the notation $\frac{\bsm_\uu}{2^{\bsj_\uu}}$ has to be interpreted as the vector $(\frac{m_i}{2^{j_i}})_{i \in \uu}$ where $\bsj_\uu = (j_i)_{i \in \uu}$, $\bsm_\uu=(m_i)_{i \in \uu}$. Hence
\begin{align*}
\frac{1}{2^m} \sum_{n=0}^{2^m-1} h_{(\bsj_\uu, -\bsone), (\bsell_\uu, \bszero)}(\bsy_n) & = \frac{1}{2^{|\bsj_\uu|/2}} \sum_{\bsk_\uu \in D_{\bsj_\uu}} \frac{1}{2^m} \sum_{n=0}^{2^m-1} \mathrm{wal}_{\bsk_\uu + 2^{\bsj_\uu}} \left(\bsy_{n,\uu} \ominus \frac{\bsell_\uu}{2^{\bsj_\uu}} \right)\\
& = \frac{1}{2^{|\bsj_\uu|/2}} \sum_{\bsk_\uu \in D_{\bsj_\uu}} \mathrm{wal}_{\bsk_\uu + 2^{\bsj_\uu}} \left(\frac{\bsell_\uu}{2^{\bsj_\uu}} \right) \frac{1}{2^m} \sum_{n=0}^{2^m-1} \mathrm{wal}_{\bsk_\uu + 2^{\bsj_\uu}} (\bsy_{n,\uu}).
\end{align*}
Therefore
\begin{equation}\label{S2}
\max_{\bsell_\uu \in D_{\bsj_\uu}} \left| \frac{1}{N} \sum_{n=0}^{N-1} h_{(\bsj_\uu, -\bsone), (\bsell_\uu, \bszero)}(\bsy_n) \right| \le  \frac{1}{2^{|\bsj_\uu|/ 2}} \sum_{\bsk_\uu \in D_{\bsj_\uu}} \left|\frac{1}{2^m} \sum_{n=0}^{2^m-1} \mathrm{wal}_{\bsk_\uu + 2^{\bsj_\uu}}(\bsy_{n,\uu})\right|.
\end{equation}
We use \eqref{S2} to estimate $S_\uu^{(1)}$ and obtain this way  
\begin{eqnarray*}
S_\uu^{(1)} & \le  & \gamma_\uu \prod_{i \in \uu} \frac{b_i - a_i}{2} \sum_{\bsj_\uu \in \{0, 1, \ldots, m-1\}^{|\uu|}} \frac{1}{2^{|\bsj_\uu|}} \sum_{\bsk_\uu \in D_{\bsj_\uu}} \left|\frac{1}{2^m} \sum_{n=0}^{2^m-1} \mathrm{wal}_{\bsk_\uu+ 2^{\bsj_\uu}}(\bsy_{n,\uu})\right|\nonumber \\
& = & \gamma_\uu \prod_{i \in \uu} (b_i - a_i) \sum_{\bsk_\uu \in \{1, \ldots, 2^m-1\}^{|\uu|}} \frac{1}{2^{\mu(\bsk_\uu)}} \left|\frac{1}{2^m} \sum_{n=0}^{2^m -1} \mathrm{wal}_{\bsk_\uu}(\bsy_{n,\uu})\right|.
\end{eqnarray*}
We illustrate the last step for the univariate case: Consider
$$\frac{b-a}{2} \sum_{j=0}^{m-1} \frac{1}{2^j} \sum_{k=0}^{2^j-1} \left|\frac{1}{2^m} \sum_{n=0}^{2^m-1} \mathrm{wal}_{k+ 2^j}(y_n)\right|$$ and substitute $k+2^j=\ell$. Then $\mu(\ell)=j+1$ and if $k$ runs through $\{0,1,\ldots,2^j-1\}$ and $j$ through $\{0,1,\ldots,m-1\}$, then $\ell$ runs through $\{1,2,\ldots,2^m -1\}$. Therefore
\begin{equation*}
\frac{b-a}{2} \sum_{j=0}^{m-1} \frac{1}{2^j} \sum_{k=0}^{2^j-1} \left|\frac{1}{2^m} \sum_{n=0}^{2^m-1} \mathrm{wal}_{k+ 2^j}(y_n)\right| = (b-a) \sum_{\ell=1}^{2^m -1} \frac{1}{2^{\mu(\ell)}} \left|\frac{1}{2^m} \sum_{n=0}^{2^m-1} \mathrm{wal}_{\ell}(y_n)\right|.
\end{equation*}
The same argumentation works for general projections $\uu$. With the definition of $E(\cP,\uu)$ in \eqref{def:epu} we obtain 
\begin{equation}\label{bd:Su1}
S_\uu^{(1)} \le  \gamma_\uu \left(\prod_{i \in \uu} (b_i - a_i)\right) E(\cP,\uu).   
\end{equation}

Now we estimate the sum $S_\uu^{(2)}$, where the summation is over all $\bsj_\uu \in \mathbb{N}_0^{|\uu|} \setminus \{0, 1, \ldots, m-1\}^{|\uu|}$. We partition the range of summation in the following way: For $\emptyset \neq \vv \subseteq \uu$ let $$B_\vv = \{ (j_i)_{i \in \uu} \in \NN_0^{|\uu|}\ :\  \mbox{$j_i \ge m$ if $i \in \vv$ and $0 \le j_i < m$ if $i \in \uu \setminus \vv$}\}.$$ Then $$\bigcup_{\emptyset \not= \vv \subseteq \uu} B_\vv = \mathbb{N}_0^{|\uu|} \setminus \{0, 1, \ldots, 2^m-1\}^{|\uu|}.$$ For $\bsj_\uu \in \mathbb{N}_0^{|\uu|} \setminus \{0, 1, \ldots, m-1\}^{|\uu|}$ and any $\bsm_\uu \in D_{\bsj_\uu}$, the set $\cP_\uu \cap I_{\bsj_\uu, \bsm_\uu}$ has at most $1$ element, since $\cP$, and therefore also $\cP_\uu$, is projection regular. Hence
\begin{equation*}
\left| \frac{1}{2^m} \sum_{n=0}^{2^m-1} h_{\bsj_u, \bsm_u}(\bsy_{n,u})  \right| \le \frac{1}{2^m}.
\end{equation*}
Then
\begin{eqnarray*}
S_\uu^{(2)} & \le & \gamma_\uu \sum_{\bsj_\uu \in \mathbb{N}_0^{|u|} \setminus \{0, 1, \ldots, 2^m-1\}^{|\uu|}} \frac{1}{2^{|\bsj_\uu| / 2}} \sum_{\stackrel{ \bsm_\uu \in D_{\bsj_\uu} }{\cP_\uu \cap I_{\bsj_\uu, \bsm_\uu} \neq \emptyset}}  \prod_{i \in \uu} \lambda_{j_i, m_i} \left| \frac{1}{2^m} \sum_{n=0}^{2^m-1} h_{(\bsj_\uu, -\bsone), (\bsm_\uu, \bszero)}(\bsy_{n,\uu}) \right| \\ 
& \le & \frac{\gamma_\uu}{2^m}  \sum_{\bsj_\uu \in \mathbb{N}_0^{|\uu|} \setminus \{0, 1, \ldots, 2^m-1\}^{|\uu|}} \frac{1}{2^{|\bsj_\uu| / 2}} \sum_{\bsm_\uu \in D_{\bsj_\uu}}  \prod_{i \in \uu} \lambda_{j_i, m_i}.
\end{eqnarray*}
Using Lemma~\ref{S1} again we get
\begin{eqnarray}\label{bd:Su2}
S_\uu^{(2)} & \le & \frac{\gamma_\uu}{2^m} \prod_{i \in \uu} \frac{b_i-a_i}{2}  \sum_{\emptyset \neq \vv \subseteq \uu} \left(\sum_{j=m}^{\infty}\frac{1}{2^{j/2}}\right)^{|\vv|} \left(\sum_{j=0}^{m-1}\frac{1}{2^{j/2}}\right)^{|\uu|-|\vv|}\nonumber\\
& = & \frac{\gamma_\uu}{2^m} \prod_{i \in \uu} \frac{b_i-a_i}{2} \left(\left(\sum_{j=0}^{\infty}\frac{1}{2^{j/2}}\right)^{|\uu|} -\left(\sum_{j=0}^{m-1} \frac{1}{2^{j/2}} \right)^{|\uu|}\right)\nonumber\\
& \le & \frac{\gamma_\uu}{2^m} \prod_{i \in \uu} \frac{b_i-a_i}{2 - \sqrt{2}}\nonumber\\
& \le & \frac{\gamma_\uu}{2^m} 2^{|\uu|} \prod_{i \in \uu} (b_i-a_i).
\end{eqnarray}

Combining the estimates \eqref{bd:Su1} and \eqref{bd:Su2} gives
$$S_\uu \le \gamma_\uu \prod_{i \in \uu} (b_i - a_i) \left(\frac{2^{|\uu|}}{2^m}  + E(\cP,|\uu|)  \right).$$
Inserting this estimate into \eqref{err_bound1} we obtain 
\begin{equation*}
|\mathrm{err}(F; \varphi; \cP)| \le \| F \|_{p,s,\bsgamma}\left(\sum_{\emptyset \neq \uu \subseteq [s]}  \gamma_\uu^q \prod_{i \in \uu} (b_i - a_i)^q \left(\frac{2^{|\uu|}}{2^m}  + E(\cP,\uu)   \right)^q \right)^{1/q}.
\end{equation*}
This yields the desired result. 
\end{proof}

In the following sections we will study the upper bound on the worst-case error from Theorem~\ref{thm1} for digital nets, digital sequences and also for polynomial lattice point sets. The basic definitions are recalled in the next section. %Experts on digital nets and sequences and polynomial lattices may jump directly to Section~\ref{sec:errdignet}.

\section{Digital nets and sequences, and polynomial lattice point sets}\label{sec:dignetseqplps}

In this section we recall the definition and basic results about the node sets in use. Reader how are already acquainted with the digital construction schemes can jump directly to Section~\ref{sec:errdignet}.

\paragraph{Digital $((t_\uu)_\uu,m,s)$-nets.} Let $\FF_2$ be the finite field of order $2$. We identify $\FF_2$ with the set $\{0,1\}$ equipped with arithmetic operations modulo 2. 

\begin{definition}\rm \label{def:dignet}
Let $s \ge 1$, $m \ge 1$ and $0 \le t \le m$ be integers. Choose
 $m \times m$ matrices $C_1,\ldots ,C_s$ over $\FF_2$ with the following property:

For any non-negative integers $d_1,\ldots ,d_s$ with $d_1+\cdots+d_s=m-t$ the system of the 
\begin{center}
\begin{tabbing}
\hspace*{3cm}\=first \= $d_1\ \ \ $ \= rows of $C_1$, \=  together with the\\
\> first \> $d_2$ \> rows of $C_2$, \=  together with the \\
\>\hspace{1cm}$\vdots$ \\
\> first \> $d_{s-1}$ \> rows of $C_{s-1}$, \=  together with the \\
\>first \> $d_s$ \>rows of $C_s$ 
\end{tabbing}
\end{center}
is linearly independent over $\FF_2$.

Consider the following construction principle for point sets consisting of $2^m$ points in $[0,1)^s$: represent $n \in \{0,1,\ldots,2^m-1\}$ in base 2, say $n=n_0+n_1 2+\cdots +n_{m-1} 2^{m-1}$, and multiply for every $i \in [s]$ the matrix $C_i$ with the vector $\vec{n} = (n_0,\ldots,n_{m-1})^{\top}$ of digits of $n$ in $\FF_2$,
\begin{eqnarray}\label{matrix_vec_net}
C_i \vec{n}=:(y_1^{(i)},\ldots ,y_m^{(i)})^{\top}.
\end{eqnarray} 
Now we set 
\begin{eqnarray}\label{defxni}
x_n^{(i)}:=\frac{y_1^{(i)}}{2}+ \cdots +\frac{y_m^{(i)}}{2^m}
\end{eqnarray}
and
\begin{eqnarray}
\bsx_n = (x_n^{(1)}, \ldots ,x_n^{(s)}).\nonumber
\end{eqnarray}
The point set $\{\bsx_0,\ldots,\bsx_{2^m-1}\}$ is called a {\it digital $(t,m,s)$-net over $\FF_2$} and the matrices $C_1,\ldots,C_s$ are called the {\it generating matrices} of the digital net.
\end{definition}

We remark that explicit constructions for digital $(t,m,s)$-nets are known with some restrictions on the parameter $t$ (the so-called quality parameter $t$ is independent of $m$ but depends on $s$), see for instance \cite{DP10,niesiam} for more information.

It is clear that every projection of a digital $(t,m,s)$-net over $\FF_2$ to coordinates from a set $\uu \subseteq [s]$, $\uu \not=\emptyset$, forms a digital $(t,m,|\uu|)$-net over $\FF_2$, see, e.g. \cite[Section~4.4.3]{DP10}. However, it may happen that the quality parameter $t$ of a projection is smaller than the overall quality parameter $t$ of the full projection. In order to include this possibility into the definition of digital nets we can define a more general form of the quality parameter $t$ in the following way.

\begin{definition}\rm
Let $C_1, \ldots, C_s$ be $m \times m$ matrices over $\mathbb{F}_2$. Then the digital net with generating matrices $C_1, \ldots, C_s$ is a digital $((t_\uu)_{\emptyset \neq \uu \subseteq [s]}, m, s)$-net over $\mathbb{F}_2$ if for all $\emptyset \neq \uu \subseteq [s]$ the matrices $(C_i)_{i \in \uu}$ generate a digital $(t_\uu, m, |\uu|)$-net over $\mathbb{F}_2$.
\end{definition}

A variant of digital nets are shifted digital nets. Here one chooses $(\vec{\sigma}_1,\ldots,\vec{\sigma}_s) \in (\mathbb{F}_2^{\mathbb{N}})^s$ with $\vec{\sigma}_i=(\sigma_1^{(i)},\sigma_2^{(i)},\ldots)^{\top} \in \mathbb{F}_2^{\mathbb{N}}$ with all but finitely many components different from zero and replaces \eqref{defxni} by
\begin{equation*}
x_n^{(i)} := \frac{y_1^{(i)}\oplus \sigma_1^{(i)}}{2} +  \cdots + \frac{y_m^{(i)}\oplus \sigma_m^{(i)}}{2^m} + \sum_{k=m+1}^{\infty}\frac{\sigma_k^{(i)}}{2^k} \in [0,1),
\end{equation*}
where $(y_1^{(i)},\ldots ,y_m^{(i)})$ are given by \eqref{matrix_vec_net}. %We denote a digitally shifted digital net by $\cP\oplus \bssigma$.

\paragraph{Digital $((t_\uu)_\uu, s)$-sequences.}

Digital sequences are infinite versions of digital nets. 

\begin{definition}\rm
Let $C_1,\ldots, C_s \in \mathbb{F}_2^{\mathbb{N} \times \mathbb{N}}$ be $\mathbb{N} \times \mathbb{N}$ matrices over $\mathbb{F}_2$. For $C_i = (c_{i,k,\ell})_{k, \ell \in \mathbb{N}}$ we assume that for each $\ell \in \mathbb{N}$ there exists a $K(\ell) \in \mathbb{N}$ such that $c_{i,k,\ell} = 0$ for all $k > K(\ell)$. Assume that for every $m \ge t$ the upper left $m \times m$ submatrices $C_1^{(m \times m)},\ldots,C_s^{(m \times m)}$ of $C_1,\ldots,C_s$, respectively, generate a digital $(t,m,s)$-net over $\FF_2$.

Consider the following construction principle for infinite sequences of points in $[0,1)^s$: represent $n \in \NN_0$ in base 2, say $n = n_0 + n_1 2 + \cdots + n_{m-1} 2^{m-1} \in \mathbb{N}_0$,  and define the infinite binary digit vector of $n$ by $\vec{n} = (n_0, n_1, \ldots, n_{m-1}, 0, 0, \ldots )^\top \in \mathbb{F}_2^{\mathbb{N}}$. For $i \in [s]$ compute
\begin{equation}\label{eq_dig_seq}
C_i \vec{n}=:(y_1^{(i)},y_2^{(i)},y_3^{(i)},\ldots)^\top, 
\end{equation}
where the matrix vector product is evaluated over $\mathbb{F}_2$. Now set
\begin{equation*}
x_n^{(i)} =  \frac{y_1^{(i)}}{2}+\frac{y_2^{(i)}}{2^2} + \frac{y_3^{(i)}}{2^3}+\cdots 
\end{equation*}
and $$\bsx_n=(x_n^{(1)},\ldots,x_n^{(s)}).$$
The infinite sequence $(\bsx_n)_{n \ge 0}$ is called a {\it digital $(t,s)$-sequence over $\FF_2$  with generating matrices $C_1,\ldots,C_s$}. 
\end{definition}

In the same way as above for digital nets we can regard a digital $(t,s)$-sequence over $\FF_2$ as a digital $((t_\uu)_{\emptyset \not= \uu \subseteq [s]},s)$-sequence over $\FF_2$.\\

For general properties of digital nets and sequences we refer to the books \cite{DP10,niesiam}.

\paragraph{Polynomial lattices.} Let $\FF_2[x]$ be the set of all polynomials over $\FF_2$ and let $\FF_2((x^{-1}))$ be the field of formal Laurent series consisting of elements $$g=\sum_{k=w}^{\infty} a_k x^{-k}\ \ \ \mbox{ with } a_k \in \FF_2 \ \mbox{ and } \ w \in \ZZ \ \mbox{ with }\ a_w \not=0.$$ For $g \in \FF_2((x^{-1}))$ and $m \in \NN \cup \{\infty\}$ we define the ``fractional part'' function $\FF_2((x^{-1})) \rightarrow [0,1)$ by $$\{g\}_{m}:= \sum_{k=\max(1,w)}^{m} a_k 2^{-k}.$$ 

Polynomial lattice point sets have been first introduced by Niederreiter~\cite{nie92} and can be viewed as polynomial analogs of lattice point sets. They form special instances of digital nets;  see~\cite[Chapter~10]{DP10} or \cite{niesiam} for further information.

\begin{definition}\label{def:plps}\rm
Let $m \in \NN$. Given a polynomial $f \in \FF_2[x]$ with $\deg(f)=m$ and $\bsg=(g_1,\ldots ,g_s) \in \FF_2[x]^s$, a {\it polynomial lattice point set} $\cP(\bsg,f)$ is given by the points $$\bsx_h=\left(\left\{\frac{h g_1}{f}\right\}_{m},\ldots ,\left\{\frac{h g_s}{f}\right\}_{m}\right)\quad \mbox{where $h \in \FF_2[x]$ with $\deg(h)<m$.}$$ QMC rules that use polynomial lattice point sets as underlying nodes are called {\it polynomial lattice rules}. The polynomial $f$ is called the modulus and $\bsg$ the generating vector of the polynomial lattice point set. Note that $|\cP(\bsg,f)|=2^m$.
\end{definition}

It is well-known that $\cP(\bsg,f)$ is projection regular whenever $\gcd(g_i,f)=1$ for all $i \in [s]$ (see, e.g., \cite[Remark~10.3]{DP10}). 

\begin{definition}\rm
The {\it dual net} of the polynomial lattice point set $\cP(\bsg,f)$ from Definition~\ref{def:plps} is defined as $$\cP^{\bot}(\bsg,f)=\{\bsk \in \FF_2[x]^s\, : \, \bsk \cdot \bsg \equiv 0 \pmod{f}\}.$$
\end{definition}

An important property of polynomial lattice point sets is that (see \cite[Lemmas~4.75 and 10.6]{DP10})
\begin{equation}\label{charprop}
\sum_{\bsx \in \cP(\bsg,f)} \wal_{\bsk}(\bsx) = \left\{
\begin{array}{ll}
 2^m & \mbox{ if }\bsk \in \cP^{\bot}(\bsg,f),\\
 0 & \mbox{ otherwise},
\end{array}\right.
\end{equation}
where we identify integers $k \in \NN_0$ with polynomials over $\FF_2$ in the natural way: $$k =\underbrace{\kappa_0+\kappa_1 2+\cdots +\kappa_{r-1} 2^{r-1}}_{\in \NN_0}  \cong  \underbrace{\kappa_0+\kappa_1 x+\cdots +\kappa_{r-1} x^{r-1}}_{\in \FF_2[x]}.$$

\section{Error bound for digital nets}\label{sec:errdignet}

In this section we study the error bound from Theorem~\ref{thm1} for digital nets. The main result of this section is:

\begin{theorem}\label{thm2}
For every digital $((t_\uu)_{\emptyset \not=\uu \subseteq[s]},m,s)$-net $\cP$ over $\FF_2$ with regular generating matrices we have 
\begin{equation*}
{\rm wce}(\cP; p,s,\bsgamma, \varphi)\le \frac{3}{2^m}\,\left(\sum_{\emptyset \neq \uu \subseteq [s]}  \left(\gamma_\uu \, 2^{|\uu|+t_\uu} \, m^{|\uu|}\,  \prod_{i \in \uu} (b_i - a_i)\right)^q \right)^{1/q},
\end{equation*}
where $1 \le p, q \le \infty$ with $\frac{1}{p}+\frac{1}{q}=1$, and with the obvious modifications if $q = \infty$.
\end{theorem}

The proof of this result relies on Theorem~\ref{thm1} and an estimate of the quantities $E(\cP,\uu)$ for digital nets. To this end, let $\cP=\{\bsy_0,\bsy_1,\ldots,\bsy_{2^m-1}\}$ be a digital $((t_\uu)_\uu,m,s)$-net over $\FF_2$ with generating matrices $C_1,\ldots,C_s$. Hence $N=|\cP|=2^m$. Throughout we assume that the generating matrices $C_i$, $i \in [s]$, are regular. This is no big restriction but it guarantees that the digital net $\cP$ is projection regular and therefore Theorem~\ref{thm1} applies. 

As is well known (see \cite[Lemma~2]{DP05a}), we have 
$$\frac{1}{2^m} \sum_{n=0}^{2^m-1} \mathrm{wal}_{\bsk_\uu}(\bsy_{n,\uu}) = 1$$ if $\bsk_\uu=(k_i)_{i \in \uu}$, where $\uu \subseteq [s]$, belongs to the dual net, i.e., if  $$\sum_{i \in \uu} C_i^{\top} \vec{k}_i=\vec{0}.$$ In all other cases the above Walsh sum equals zero. Hence the quantity $E(\cP,\uu)$ from \eqref{def:epu} boils down to $$E(\cP,\uu)=\sum_{\bsk_\uu \in \{1, \ldots, 2^m-1\}^{|\uu|} \atop \sum_{i \in \uu} C_i^{\top} \vec{k}_i =\vec{0}} \frac{1}{2^{\mu(\bsk_\uu)}}.$$
Now we estimate $E(\cP,\uu)$.  

\begin{lemma}\label{le:Ebound}
Let $\emptyset \neq \uu \subseteq [s]$ and let $C_i$, $i \in \uu$, be the generating matrices of a digital
$(t_\uu,m,s)$-net $\cP$ over $\FF_2$. Then we have
$$E(\cP,\uu) \le  2^{|\uu|+t_\uu+1} \, \frac{m^{|\uu|}}{2^m}.$$
\end{lemma}

\begin{proof} 
To simplify the notation we show the result only for $\uu = [s]$. The other cases follow by the same arguments. We have
\begin{eqnarray*}
E(\cP,[s]) & = & \frac{1}{2^s}\sum_{j_1,\ldots
,j_s=0}^{m-1}\frac{1}{2^{j_1+\cdots+j_s}}
\underbrace{\sum_{k_1=2^{j_1}}^{2^{j_1+1}-1}\ldots \sum_{k_s=2^{j_s}}^{2^{j_s+1}-1}}_{C_1^{\top}\vec{k}_1+\cdots +C_s^{\top}
\vec{k}_s=\vec{0}}1.
\end{eqnarray*}
For 
\begin{eqnarray}\label{cond1}
\Sigma(j_1,\ldots,j_s):=\underbrace{\sum_{k_1=2^{j_1}}^{2^{j_1+1}-1}\ldots \sum_{k_s=2^{j_s}}^{2^{j_s+1}-1}}_{C_1^{\top}\vec{k}_1+\cdots +C_s^{\top}
\vec{k}_s=\vec{0}}1
\end{eqnarray}
we know from \cite[Proof of Lemma~7]{DP05b} that $\Sigma(j_1,\ldots,j_s)=0$ if $j_1+\cdots +j_s \le m-t-s$. If 
\begin{eqnarray}\label{bed1}
j_1+\cdots +j_s \ge m-t-s+1,
\end{eqnarray}
then
$$\Sigma(j_1,\ldots,j_s) \le \left\{
\begin{array}{ll}
1 & \mbox{ if } j_1+\cdots +j_s \le m-t,\\ 
2^{j_1+\cdots+j_s-m+t} & \mbox{ if } j_1+\cdots +j_s > m-t.
\end{array}
\right.$$

Together with condition \eqref{bed1} we obtain
\begin{eqnarray}\label{sum12}
E(\cP,[s]) & \le & \frac{1}{2^s} \sum_{j_1,\ldots ,j_s=0 \atop m-t-s+1 \le j_1+\cdots
+j_s \le m-t}^{m-1} \frac{1}{2^{j_1+\cdots +j_s}}  \nonumber \\
& & + \frac{1}{2^s}\sum_{j_1,\ldots ,j_s=0 \atop  j_1+\cdots
+j_s > m-t}^{m-1} \frac{1}{2^{j_1+\cdots +j_s}} 2^{j_1+\cdots
+j_s-m+t} \nonumber \\
&=: & \Sigma_{1} +  \Sigma_{2}. 
\end{eqnarray}
Now we have to estimate the sums $\Sigma_{1}$ and
$\Sigma_{2}$. First we have
\begin{equation*}
\Sigma_{2} = \frac{1}{2^s} \frac{2^t}{2^m} \sum_{l=m-t+1}^{s
(m-1)} \sum_{j_1,\ldots ,j_s=0\atop j_1+\cdots
+j_s=l}^{m-1}1 =  \frac{1}{2^s} \frac{2^t}{2^m}
\sum_{l=m-t+1}^{s(m-1)} {l+s-1 \choose s-1}, 
\end{equation*}
where we used the fact that for fixed $l$ the number of non-negative integer
solutions of $j_1+\cdots +j_s=l$ is given by ${l+s-1 \choose
s-1}$. Using the formula $$\sum_{k=0}^m {n+k \choose n}= {n+m+1 \choose n+1}$$ we obtain $$\sum_{l=m-t+1}^{s(m-1)} {l+s-1 \choose s-1} = {sm \choose s} - {m+s-t \choose s}$$ and hence $$\Sigma_2 = \frac{1}{2^s} \frac{2^t}{2^m} \left[ {sm \choose s} - {m+s-t \choose s}\right].$$ We have $${sm \choose s} - {m+s-t \choose s} \le {sm \choose s} =\frac{sm(sm-1)\cdots (sm-(s-1))}{s!} \le \frac{s^s}{s!} m^s.$$ Hence  $$\Sigma_2 \le 2^t \frac{s^s}{2^s s!} \frac{m^s}{2^m}.$$

Now we estimate $\Sigma_{1}$. If $m-t \geq s-1$ we proceed similar to above and obtain 
\begin{eqnarray}\label{sum2_asym}
\Sigma_{1} & = & \frac{1}{2^s} \sum_{j_1,\ldots ,j_s=0 \atop m-t-s+1 \le j_1+\cdots
+j_s \le m-t}^{m-1} \frac{1}{2^{j_1+\cdots +j_s}} \nonumber \\
& = & \frac{1}{2^s} \sum_{l=m-t-s+1}^{m-t} {l+s-1 \choose
s-1} \frac{1}{2^l} \nonumber\\
& \le & \frac{1}{2^s}\frac{1}{2^{m-t-s+1}}{ m-t \choose
s-1} \left(\frac{1}{2}\right)^{-s} \nonumber\\
& \le & 2^{s-1} \frac{2^t}{2^m}  \frac{(m-t)^{s-1}}{(s-1)!}.
\end{eqnarray}
Here we used the estimate $$\sum_{t=t_0}^{\infty} {t+k-1 \choose k-1} \frac{1}{b^t} \le \frac{1}{b^{t_0}} {t_0+k-1 \choose k-1} \left(1-\frac{1}{b}\right)^{-k}$$ for integers $b>1$ and $k,t_0>0$ (see, e.g., \cite{DP05b}).

For this case ($m-t \ge s-1$) we obtain
\begin{eqnarray}\label{Ebound}
E(\cP,[s]) & \le & 2^{s-1}\, \frac{2^t}{2^m}  \frac{(m-t)^{s-1}}{(s-1)!} + 2^t\, \frac{s^s}{2^s s!} \frac{m^s}{2^m} \nonumber \\
& \leq & 2^t\, \frac{m^s}{2^m} \left( \left(\frac{2 {\rm e}}{s-1}\right)^{s-1} + \left(\frac{{\rm e}}{2}\right)^s \right) \nonumber \\
& \le &  2^{s+t+1}\, \frac{m^s}{2^m}, 
\end{eqnarray}
where we used that $$\left(\frac{2 {\rm e}}{s-1}\right)^{s-1} + \left(\frac{{\rm e}}{2}\right)^s \le 2^{s+1}$$ for  $s>1$. It can be easily checked that the bound \eqref{Ebound} also holds true also for $s=1$. 

Now we consider the case where $m-t < s-1$. We have
\begin{eqnarray}\label{sum2_asym2}
\Sigma_{1} &=& \frac{1}{2^s} \sum_{l=0}^{m-t} {l+s-1 \choose s-1} \frac{1}{2^l}  \leq  \frac{1}{2^s} \sum_{l=0}^\infty {l+s-1 \choose s-1} \frac{1}{2^l}  =   1< 2^{s-1} \frac{2^t}{2^m}.\nonumber
\end{eqnarray}
Thus we obtain
\begin{eqnarray*}
        E(\cP,[s])
& \le &  2^{s-1}\, \frac{2^t}{2^m} + 2^t \, \frac{s^s}{2^s s!} \frac{m^s}{2^m}\\
& = &  2^t\, \frac{m^s}{2^m} \left(\frac{2^{s-1}}{m^s}+\frac{s^s}{2^s s!}\right),\\
& \le &  2^{s+t-1} \frac{m^s}{2^m},
\end{eqnarray*}
where we used similar arguments as above. This finishes the proof.
\end{proof}

Theorem~\ref{thm2} follows from combining Theorem~\ref{thm1} and Lemma~\ref{le:Ebound}.

\section{Error bound for digital sequences}\label{sec:EbdDS}

In this section we study the error bound from Theorem~\ref{thm1} for digital sequences. Throughout we restrict ourselves to non-singular upper triangular generating matrices $C_1,\ldots,C_s$. Our main result of this section is:

\begin{theorem}\label{thm3}
Let $\cP_N$ be the initial segment of a digital $((t_\uu)_{\emptyset\not=\uu \subseteq [s]},s)$-sequence over $\FF_2$ consisting of the first $N$ terms. Assume that the generating matrices $C_1,\ldots,C_s$ are non-singular upper triangular matrices. Then for any $N \in \mathbb{N}$ with $N\ge 2$ we have
\begin{equation*}
{\rm wce}(\cP_N; p,s,\bsgamma, \varphi)  \le \frac{1}{N} \frac{3 \log (2N)}{\log 2}   \left(\sum_{\emptyset \neq \uu \subseteq [s]}  \left(\gamma_\uu \, 2^{t_\uu} \, (3 \log N)^{|\uu|}\,  \prod_{i \in \uu} (b_i - a_i)\right)^q \right)^{1/q}.
\end{equation*}
\end{theorem}

\begin{proof}
Let $N \in \NN$ be of the form $N=2^{m_1}+2^{m_2}+\cdots+2^{m_r}$ with integers $m_r> \ldots > m_2 >m_1 \ge 0$. Let $\cP_N=\{\bsy_0,\bsy_1,\ldots,\bsy_{N-1}\}$ be the initial segment of a digital sequence and let, for $j \in [r]$, $$\cQ_j=\{\bsy_{2^{m_1}+\cdots+2^{m_{j-1}}},\bsy_{2^{m_1}+\cdots+2^{m_{j-1}}+1},\ldots,\bsy_{2^{m_1}+\cdots+2^{m_{j}}-1}\}.$$ 
Then $$\cP_N=\bigcup_{j=1}^r \cQ_j.$$ Now 
\begin{eqnarray*}
|{\rm err}(F; \varphi; \cP_N)| & = & \left|\frac{1}{N}\sum_{n=0}^{N-1}F(\Phi^{-1}(\bsy_n))- \int_{[\bsa, \bsb]} F(\bsx) \varphi(\bsx)\rd \bsx \right|\\
& = & \left|\sum_{j=1}^r \frac{2^{m_j}}{N} \frac{1}{2^{m_j}} \sum_{n=0}^{2^{m_j}-1}F(\Phi^{-1}(\bsy_{2^{m_1}+\cdots+2^{m_{j-1}}+n}))- \sum_{j=1}^r \frac{2^{m_j}}{N} \int_{[\bsa, \bsb]} F(\bsx) \varphi(\bsx)\rd \bsx \right|\\
& \le & \sum_{j=1}^r \frac{2^{m_j}}{N}  \left|\frac{1}{2^{m_j}} \sum_{n=0}^{2^{m_j}-1}F(\Phi^{-1}(\bsy_{2^{m_1}+\cdots+2^{m_{j-1}}+n}))-  \int_{[\bsa, \bsb]} F(\bsx) \varphi(\bsx)\rd \bsx \right|\\
& = & \sum_{j=1}^r \frac{2^{m_j}}{N} |{\rm err}(F; \varphi; \cQ_j)|.
\end{eqnarray*}
From this we immediately obtain $${\rm wce}(\cP_N; p,s,\bsgamma, \varphi) \le \sum_{j=1}^r \frac{2^{m_j}}{N} {\rm wce}(\cQ_j; p,s,\bsgamma, \varphi).$$

Now we claim that the point sets $\cQ_j$, $j \in [r]$, are shifted digital $((t_\uu)_{\emptyset \not=\uu \subseteq [s]},m_j,s)$-nets over $\FF_2$.

Indeed, let $C_1,\ldots,C_s \in \FF_2^{\NN \times \NN}$ be the generating matrices of the digital sequence. According to our standing assumption these matrices are non-singular upper triangular matrices. For $i \in [s]$ the matrix $C_i$ is of the form
$$
C_i = \left( \begin{array}{ccc}
           & \vline &  \\
  C_i^{(m_j \times m_j)} & \vline & D_i^{(m_j \times \mathbb{N})} \\
           & \vline &   \\ \hline
    & \vline & \\
 0^{(\mathbb{N} \times m_j)}  & \vline &  V_i^{(\mathbb{N} \times \mathbb{N})} \\
   &   \vline       &
\end{array} \right) \in \mathbb{F}_2^{\mathbb{N} \times \mathbb{N}},
$$ 
where $C_i^{(m_j \times m_j)}$ is the upper-left $m_j \times m_j$ sub-matrix of $C_i$, $0^{(\mathbb{N} \times m_j)}$ denotes the $\mathbb{N} \times m_j$ zero matrix. Furthermore, also the matrix $V_i^{(\mathbb{N} \times \mathbb{N})}$ is non-singular upper triangular like $C_i$. 

Any integer $k \in \{2^{m_1}+\cdots + 2^{m_{j -1}}, 2^{m_1}+\cdots + 2^{m_{j-1}} + 1, \ldots ,2^{m_1} + \cdots + 2^{m_j}-1\}$ can be written in the form $$k=2 ^{m_1}+\cdots+2^{m_{j-1}}+a=2^{m_{j-1}} \ell +a$$ with $a \in \{0,1,\ldots,2^{m_j}-1\}$ and $\ell=1+2^{m_{j}-m_{j-1}}+\cdots+2^{m_1-m_{j-1}}$ if $j > 1$ and $\ell =0$ if $j=1$. Hence the dyadic digit vector of $k$ is given by $$\vec{k}=(a_0,a_1,\ldots,a_{m_j-1},l_0,l_1,l_2,\ldots)^{\top}=:{\vec{a} \choose \vec{\ell}},$$ where $a_0,\ldots,a_{m_j-1}$ are the binary digits of $a$ and $l_0,l_1,l_2,\ldots$ are the dyadic digits of $\ell$. Note that the elements of the sequence $l_0,l_1,l_2,\ldots$ become eventually zero. With this notation and with the above decomposition of the matrix $C_i$  we have
$$C_i \vec{k}=\left( \begin{array}{c} C_i^{(m_j \times m_j)}  \vec{a} \\ 0 \\ 0 \\ \vdots \end{array} \right)  + \left( \begin{array}{c}
             \\
  D_i^{(m_j \times \mathbb{N})}  \\
          \\ \hline  \\
 V_i^{(\mathbb{N} \times \mathbb{N})}  \\
   \end{array} \right) \vec{\ell}.$$ 
For the point set $\cQ_j$ under consideration, the vector
\begin{equation}\label{dig_shift}
\vec{\sigma}_{j,i}:=\left( \begin{array}{c}
             \\
  D_i^{(m_j \times \mathbb{N})}  \\
          \\ \hline  \\
 V_i^{(\mathbb{N} \times \mathbb{N})}  \\
   \end{array} \right) \vec{\ell}
\end{equation}
is constant and its components become eventually zero (i.e., only a finite number of components is nonzero). Furthermore, $C_i^{(m_j \times m_j)}  \vec{a}$ for $a=0,1,\ldots,2^{m_j}-1$ and $i \in [s]$ generate a digital $(t,m_j,s)$-net over $\mathbb{F}_2$.

This means that the point set $\cQ_j$ is a digitally shifted $((t_\uu)_{\emptyset \not=\uu \subseteq[s]},m_j,s)$-net over $\mathbb{F}_2$ with generating matrices $C_1^{(m_j \times m_j)},\ldots, C_s^{(m_j \times m_j)}$ and hence the claim is proven.\\

Using Theorem~\ref{thm1} and Theorem~\ref{thm2} (the result also holds for digitally shifted digital nets) we obtain
\begin{equation*}
{\rm wce}(\cQ_j; p,s,\bsgamma, \varphi) \le \frac{3}{2^{m_j}}\,\left(\sum_{\emptyset \neq \uu \subseteq [s]}  \left(\gamma_\uu \, 2^{|\uu|+t_\uu} \, m_j^{|\uu|}\,  \prod_{i \in \uu} (b_i - a_i)\right)^q \right)^{1/q}
\end{equation*}
and therefore 
\begin{equation*}
{\rm wce}(\cP_N; p,s,\bsgamma, \varphi) \le \frac{3}{N}\sum_{j=1}^r \left(\sum_{\emptyset \neq \uu \subseteq [s]}  \left(\gamma_\uu \, 2^{|\uu|+t_\uu} \, m_j^{|\uu|}\,  \prod_{i \in \uu} (b_i - a_i)\right)^q \right)^{1/q}.
\end{equation*}
For $j \in [r]$ we have $2^{m_j}\le 2^{m_1}+2^{m_2}+\cdots + 2^{m_r}=N$ and this implies $m_j \le (\log N)/\log 2$. Furthermore, because of $m_r > \ldots > m_2 >m_1\ge 0$, we have $m_r \ge r-1$ and hence $r \le 1+ (\log N)/(\log 2) =\log(2N)/\log 2$. Hence we obtain
\begin{eqnarray*}
{\rm wce}(\cP_N; p,s,\bsgamma, \varphi) & \le & \frac{3 r}{N}  \left(\sum_{\emptyset \neq \uu \subseteq [s]}  \left(\gamma_\uu \, 2^{|\uu|+t_\uu} \, \left(\frac{\log N}{\log 2}\right)^{|\uu|}\,  \prod_{i \in \uu} (b_i - a_i)\right)^q \right)^{1/q}\\
& \le & \frac{1}{N} \frac{3 \log (2N)}{\log 2}   \left(\sum_{\emptyset \neq \uu \subseteq [s]}  \left(\gamma_\uu \, 2^{t_\uu} \, (3 \log N)^{|\uu|}\,  \prod_{i \in \uu} (b_i - a_i)\right)^q \right)^{1/q},
\end{eqnarray*}
where we estimated $2/\log 2=2.8853\ldots < 3$ in order to make the result a bit easier to state.
\end{proof}

\section{Error bound for polynomial lattice point sets}\label{sec:PLPS}

In this section we analyze the worst-case error of polynomial lattice points $\cP(\bsg,f)$ as introduced in Definition~\ref{def:plps}. We restrict our analysis to the case $p=\infty$ and hence $q=1$. Then Theorem~\ref{thm1} yields
\begin{equation}\label{wce:estPLPS}
{\rm wce}(\cP; \infty,s,\bsgamma, \varphi)\le \frac{1}{2^m} \sum_{\emptyset \neq \uu \subseteq [s]}  \gamma_\uu \, 2^{|\uu|}\, \prod_{i \in \uu} (b_i - a_i) + \sum_{\emptyset \neq \uu \subseteq [s]}  \gamma_\uu \, \prod_{i \in \uu} (b_i - a_i) \,  E(\cP,\uu).
\end{equation}
For a polynomial lattice point set $\cP(\bsg,f)$ we use $$B_{\bsgamma}(\bsg,f) := \sum_{\emptyset \neq \uu \subseteq [s]}  \gamma_\uu \, \prod_{i \in \uu} (b_i - a_i) \,  E(\cP(\bsg,f),\uu)$$
 as figure-of-merit.
 
According to \eqref{charprop} we have $$\frac{1}{2^m} \sum_{n=0}^{2^m-1} \mathrm{wal}_{\bsk_\uu}(\bsy_{n,\uu}) =\left\{ 
\begin{array}{ll}
1 & \mbox{ if } \  \bsg_\uu \cdot \bsk_\uu \equiv 0 \pmod{f},\\
0 & \mbox{ otherwise}.
\end{array}\right.$$
Hence
\begin{equation}\label{def:EuPLPS}
E(\cP(\bsg,f),\uu) =   \sum_{\bsk_\uu \in \{1, \ldots, 2^m-1\}^{|\uu|} \atop \bsg_\uu \cdot \bsk_\uu \equiv 0 \pmod{f}} \frac{1}{2^{\mu(\bsk_\uu)}}. 
\end{equation}

\paragraph{Existence results.} For given irreducible polynomial $f \in \FF_2[x]$ with $\deg(f)=m$ we average $B_{\bsgamma}(\bsg,f)$ over all possible generating vectors $\bsg \in G_m^s$, where $$G_m=\{h \in \FF_2[x]\ : \ \deg(h)<m\}.$$ 

\begin{theorem}\label{le_av}
For any irreducible polynomial $f \in \FF_2[x]$ with $\deg(f)=m$ we have $$\frac{1}{|G_m^s|} \sum_{\bsg \in G_m^s}B_{\bsgamma}(\bsg,f) =   \frac{1}{2^m} \sum_{\emptyset \neq \uu \subseteq [s]}\gamma_\uu \left(\frac{m}{2}\right)^{|\uu|} \prod_{i \in \uu} (b_i-a_i).$$ Furthermore, for every $c \ge 1$ we have
\begin{equation}\label{eq:Markov}
\#\left\{ \bsg \in G_m^s \ : B_{\bsgamma}(\bsg,f) \le  \frac{c}{2^m} \sum_{\emptyset \neq \uu \subseteq [s]}\gamma_\uu \left(\frac{m}{2}\right)^{|\uu|} \prod_{i \in \uu} (b_i-a_i) \right\} > 2^{s m} \left(1-\frac{1}{c}\right).
\end{equation}
In particular, there exists a $\bsg^\ast \in G_m^s$ such that 
\begin{equation}\label{ex:qast}
B_{\bsgamma}(\bsg^\ast,f) \le  \frac{1}{2^m} \sum_{\emptyset \neq \uu \subseteq [s]}\gamma_\uu \left(\frac{m}{2}\right)^{|\uu|} \prod_{i \in \uu} (b_i-a_i).
\end{equation}
\end{theorem}

\begin{proof}
The proof uses standard techniques. Note that $|G_m^s|=2^{s m}$. We have
\begin{eqnarray*}
\lefteqn{\frac{1}{|G_m^s|} \sum_{\bsg \in G_m^s}B_{\bsgamma}(\bsg,f)}\\
& = & \frac{1}{2^{sm}} \sum_{\bsg \in G_m^s} \sum_{\emptyset \neq \uu \subseteq [s]}  \gamma_\uu \, \prod_{i \in \uu} (b_i - a_i)  \sum_{\bsk_\uu \in \{1, \ldots, 2^m-1\}^{|\uu|} \atop \bsg_\uu \cdot \bsk_\uu \equiv 0 \pmod{f}} \frac{1}{2^{\mu(\bsk_\uu)}}\\
& = & \sum_{\emptyset \neq \uu \subseteq [s]}  \gamma_\uu \, \prod_{i \in \uu} (b_i - a_i) \,  \sum_{\bsk_\uu \in \{1, \ldots, 2^m-1\}^{|\uu|}} \frac{1}{2^{\mu(\bsk_\uu)}}  \frac{1}{2^{
|\uu| m}} \sum_{\bsg_\uu \in G_m^{|\uu|} \atop \bsg_\uu \cdot \bsk_\uu \equiv 0 \pmod{f}} 1\\
& = & \frac{1}{2^m} \sum_{\emptyset \neq \uu \subseteq [s]}  \gamma_\uu \, \prod_{i \in \uu} (b_i - a_i) \,  \sum_{\bsk_\uu \in \{1, \ldots, 2^m-1\}^{|\uu|}} \frac{1}{2^{\mu(\bsk_\uu)}},
\end{eqnarray*}
where we used that $$\sum_{\bsg_\uu \in G_m^{|\uu|} \atop \bsg_\uu \cdot \bsk_\uu \equiv 0 \pmod{f}} 1=(2^m)^{|\uu|-1},$$ because $f$ is irreducible.

We have $$ \sum_{\bsk_\uu \in \{1, \ldots, 2^m-1\}^{|\uu|}} \frac{1}{2^{\mu(\bsk_\uu)}}=\left(\sum_{k=1}^{2^m -1} \frac{1}{2^{\mu(k)}}\right)^{|\uu|}$$ and 
\begin{equation}\label{sum:muk}
\sum_{k=1}^{2^m -1} \frac{1}{2^{\mu(k)}} =\sum_{a=0}^{m-1}\sum_{k=2^a}^{2^{a+1}-1} \frac{1}{2^{a+1}} =\frac{m}{2}.
\end{equation}
This yields $$\frac{1}{|G_m^s|} \sum_{\bsg \in G_m^s}B_{\bsgamma}(\bsg,f) =   \frac{1}{2^m} \sum_{\emptyset \neq \uu \subseteq [s]}\gamma_\uu \left(\frac{m}{2}\right)^{|\uu|} \prod_{i \in \uu} (b_i-a_i),$$ as desired.

The estimate \eqref{eq:Markov} follows from an application of Markov's inequality. The existence of $\bsg^\ast$ satisfying \eqref{ex:qast} follows from choosing $c=1$ in \eqref{eq:Markov}.
\end{proof}

Combining the above existence result with \eqref{wce:estPLPS} we obtain the following corollary:

\begin{corollary}
For any irreducible polynomial $f \in \FF_2[x]$ with $\deg(f)=m$ there exists a $\bsg^\ast \in G_m^s$ such that 
\begin{equation*}
{\rm wce}(\cP(\bsg^\ast,f); \infty,s,\bsgamma, \varphi)\le \frac{1}{2^m} \sum_{\emptyset \neq \uu \subseteq [s]}  \gamma_\uu \, \left(2^{|\uu|} + \left(\frac{m}{2}\right)^{|\uu|} \right)\, \prod_{i \in \uu} (b_i - a_i).
\end{equation*}
\end{corollary}

\paragraph{Fast component-by-component construction.}

Initially developed for lattice rules (see, e.g., \cite{K59, K03,NC06,SKJ02}), nowadays the fast component-by-component construction is also used to construct polynomial lattice rules (see, e.g.,  \cite{DKPS05,DLP05}). This construction method is very efficient and yields polynomial lattice rules which guarantee the almost optimal order of error convergence in the number of employed nodes and even the currently best results in terms of the dependence of the error on the dimension $s$.

Here we propose a component-by-component algorithm that is based on the figure-of-merit $B_{\bsgamma}$.

\begin{algorithm}\label{alg_weight}
Given a polynomial $f\in \FF_2[x]$, with $\deg(f) =m$, and weights $\bsgamma = \{\gamma_\uu \in \mathbb{R}_+ \ : \ \uu \subseteq [s]\}$. Construct a vector $\bsg^\ast=(g_1^\ast,\ldots,g_s^\ast)$ as follows:
\begin{enumerate}
\item Set $g_1^*=1$.
\item For $d=2,3,\ldots,s$: if $g_1^\ast,\ldots,g_{d-1}^\ast$ are already chosen, then find $g_d^* \in G_m$ by minimizing $B_{\bsgamma}((g_1^*,\ldots ,g_{d-1}^*,g_d),f)$ as a function of $g_d$ over $G_m$.
\end{enumerate}
\end{algorithm}

\begin{remark}\rm 
Note that the figure-of-merit $B_{\bsgamma}$ depends only on the size of the cube but is otherwise independent of the product measure. Hence, the polynomial lattice rule constructed by Algorithm~\ref{alg_weight} is universal in the sense that it works for any product $\varphi$ of PDFs. \end{remark}

In the following we restrict ourselves again to irreducible polynomials $f \in \mathbb{F}_2[x]$ and show that in this case the resulting generating vector $\bsg^\ast$ is of good quality with respect to the figure-of-merit $B_{\bsgamma}$. The resulting polynomial lattice point set $\cP(\bsg^*,f)$ is projection regular, because $f$ is irreducible and $\bsg^* \in G_m$ and hence $\gcd(g_i^*,f)=1$ for every $i \in [s]$.

\begin{theorem}\label{thm:CBC}
Let  $f\in \FF_2[x]$ be irreducible with $\deg(f) =m$. Suppose that $\bsg^*=(g_1^*,\ldots,g_s^*) \in G_m ^s$ is constructed according to Algorithm~\ref{alg_weight}. Then for all $d \in [s]$ we have  
$$B_{\bsgamma}((g_1^*,\ldots,g_d^*),f) \le \frac{1}{2^m}\sum_{\emptyset \not=\uu \subseteq [d]} \gamma_\uu \left(\frac{m}{2}\right)^{|\uu|} \prod_{i \in \uu} (b_i-a_i).$$
\end{theorem}

\begin{proof}
The proof uses standard arguments and is based on induction on $d$.

For $d=1$ we have $B_{\bsgamma}((1),f)=0$ and hence the desired bound holds true trivially.

In the induction hypothesis we assume that for some $d \in \{2,\ldots,s\}$ we have already constructed $\bsg_{d-1}^{\ast}:=(g_1^*,\ldots,g_{d-1}^*) \in G_m^{d-1}$ such that 
\begin{equation}\label{indhyp}
B_{\bsgamma}(\bsg_{d-1}^*,f)\le \frac{1}{2^m}\sum_{\emptyset \not=\uu \subseteq [d-1]} \gamma_{\uu} \left(\frac{m}{2}\right)^{|\uu|} \prod_{i \in \uu} (b_i-a_i).
\end{equation}

Now we  consider $B_{\bsgamma}((\bsg_{d-1}^*,g_d),f)$ where $(\bsg_{d-1}^*,g_d):=(g_1^*,\ldots,g_{d-1}^*,g_d)$. Separating the summation over all $\emptyset \not=\uu \subseteq [d]$ into summation over all $\emptyset \not=\uu \subseteq [d-1]$ and summation over all $\uu \subseteq [d]$ that contain the element $d$ we obtain
\begin{eqnarray}\label{eq:indstep}
B_{\bsgamma}((\bsg_{d-1}^*,g_d),f) & = & \sum_{\emptyset \neq \uu \subseteq [d]} \gamma_{\uu} \prod_{i \in \uu} (b_i -a_i) \sum_{\bsk_\uu \in \{1, \ldots, 2^m-1\}^{|\uu|} \atop (\bsg_{d-1}^*,g_d)_\uu \cdot \bsk_\uu \equiv 0 \pmod{f}} \frac{1}{2^{\mu(\bsk_\uu)}}\nonumber \\
& = & B_{\bsgamma}(\bsg_{d-1}^*,f)+\Theta(g_d),
\end{eqnarray}
where
\begin{eqnarray*}
\Theta(g_d) & := & \sum_{d\in \uu \subseteq [d]} \gamma_\uu \prod_{i \in \uu} (b_i-a_i)  \sum_{\bsk_\uu \in \{1, \ldots, 2^m-1\}^{|\uu|} \atop (\bsg_{d-1}^*,g_d)_\uu \cdot \bsk_\uu \equiv 0 \pmod{f}} \frac{1}{2^{\mu(\bsk_\uu)}}\\
& = & \sum_{d\in \uu \subseteq [d]} \gamma_\uu \prod_{i \in \uu} (b_i-a_i) \sum_{k_d = 1}^{2^m -1} \frac{1}{2^{\mu(k_d)}} \sum_{\bsk_{\uu\setminus \{d\}} \in \{1, \ldots, 2^m-1\}^{|\uu|-1} \atop (\bsg_{d-1}^*)_{\uu \setminus \{d\}} \cdot \bsk_{\uu \setminus \{d\}} + g_d k_d \equiv 0 \pmod{f}} \frac{1}{2^{\mu(\bsk_{\uu\setminus \{d\}})}}.
\end{eqnarray*}
Note that the only dependence of $B_{\bsgamma}((\bsg_{d-1}^*,g_d),f)$ on $g_d$ is via $\Theta(g_d)$ and hence a minimizer of $B_{\bsgamma}((\bsg_{d-1}^*,g_d),p)$ is also a minimizer of $\Theta(g_d)$ and vice versa.

Now we use an averaging argument. Since the minimum never exceeds the average we have 
\begin{eqnarray*}
 \Theta(g_d^*) & \le & \frac{1}{2^m} \sum_{g_d \in G_m} \Theta(g_d)\\
 & = & \frac{1}{2^m} \sum_{g_d \in G_m}  \sum_{d\in \uu \subseteq [d]} \gamma_\uu \prod_{i \in \uu} (b_i-a_i) \sum_{k_d = 1}^{2^m -1} \frac{1}{2^{\mu(k_d)}} \sum_{\bsk_{\uu\setminus \{d\}} \in \{1, \ldots, 2^m-1\}^{|\uu|-1} \atop (\bsg_{d-1}^*)_{\uu \setminus \{d\}} \cdot \bsk_{\uu \setminus \{d\}} + g_d k_d \equiv 0 \pmod{f}} \frac{1}{2^{\mu(\bsk_{\uu\setminus \{d\}})}}\\
& = & \frac{1}{2^m} \sum_{d\in \uu \subseteq [d]} \gamma_\uu \prod_{i \in \uu} (b_i - a_i) \sum_{k_d = 1}^{2^m -1} \frac{1}{2^{\mu(k_d)}} \sum_{\bsk_{\uu\setminus \{d\}} \in \{1, \ldots, 2^m-1\}^{|\uu|-1}} \frac{1}{2^{\mu(\bsk_{\uu\setminus \{d\}})}}\\
&& \hspace{2cm} \times \sum_{g_d \in G_m \atop g_d k_d \equiv -(\bsg_{d-1}^*)_{\uu \setminus \{d\}} \cdot \bsk_{\uu \setminus \{d\}} \pmod{f}} 1.
\end{eqnarray*}
Since, for fixed $k_d \in \{1,\ldots,2^m -1\}$, the polynomial $\gcd(k_d,f)=1$, because $f$ is irreducible and $\deg(k_d) < \deg(f)$, it follows that the linear polynomial congruence $$g_d k_d \equiv -(\bsg_{d-1}^*)_{\uu \setminus \{d\}} \cdot \bsk_{\uu \setminus \{d\}} \pmod{f}$$ has exactly one solution $g_d \in G_{m}$. Hence
\begin{eqnarray*}
 \Theta(g_d^*) & \le & \frac{1}{2^m} \sum_{d\in \uu \subseteq [d]} \gamma_\uu \prod_{i \in \uu} (b_i - a_i) \sum_{k_d = 1}^{2^m -1} \frac{1}{2^{\mu(k_d)}} \sum_{\bsk_{\uu\setminus \{d\}} \in \{1, \ldots, 2^m-1\}^{|\uu|-1}} \frac{1}{2^{\mu(\bsk_{\uu\setminus \{d\}})}}\\
& = & \frac{1}{2^m} \sum_{d\in \uu \subseteq [d]} \gamma_\uu \prod_{i \in \uu} (b_i - a_i) \left(\sum_{k = 1}^{2^m -1} \frac{1}{2^{\mu(k)}}\right)^{|\uu|} \\
& = &  \frac{1}{2^m} \sum_{d\in \uu \subseteq [d]} \gamma_\uu \, \left(\frac{m}{2}\right)^{|\uu|} \,  \prod_{i \in \uu} (b_i - a_i), 
\end{eqnarray*}
according to \eqref{sum:muk}. Inserting this into \eqref{eq:indstep} and using the induction hypothesis \eqref{indhyp} we obtain
\begin{eqnarray*}
B_{\bsgamma}((\bsg_{d-1}^*,g_d^*),f) & \le &  \frac{1}{2^m}\sum_{\emptyset \not= \uu \subseteq [d-1]} \gamma_\uu \, \left(\frac{m}{2}\right)^{|\uu|} \, \prod_{i \in \uu} (b_i - a_i) + \frac{1}{2^m} \sum_{d\in \uu \subseteq [d]} \gamma_\uu \, \left(\frac{m}{2}\right)^{|\uu|} \,  \prod_{i \in \uu} (b_i - a_i) \\
& = & \frac{1}{2^m}\sum_{\emptyset \not= \uu \subseteq [d]} \gamma_\uu \, \left(\frac{m}{2}\right)^{|\uu|} \, \prod_{i \in \uu} (b_i - a_i).
\end{eqnarray*}
This finishes the induction proof of Theorem~\ref{thm:CBC}. 
\end{proof}

Combining \eqref{wce:estPLPS} and Theorem~\ref{thm:CBC} we obtain:

\begin{corollary}\label{co1}
Let  $f\in \FF_2[x]$ be irreducible with $\deg(f) =m$. Suppose that $\bsg^*=(g_1^*,\ldots,g_s^*) \in G_m^s$ is constructed according to Algorithm~\ref{alg_weight}. Then for all $d\in[s]$ we have  
\begin{equation*}
{\rm wce}(\cP((g_1^*,\ldots,g_d^*),f); \infty,d,\bsgamma, \varphi)\le \frac{1}{2^m} \sum_{\emptyset \neq \uu \subseteq [d]}  \gamma_\uu \, \left(2^{|\uu|} + \left(\frac{m}{2}\right)^{|\uu|} \right)\, \prod_{i \in \uu} (b_i - a_i).
\end{equation*}
\end{corollary}

\paragraph{Efficient calculation of $B_{\bsgamma}(\bsg,f)$.}

In order to be able to implement the fast component-by-component construction, we need an efficient method for calculating the quality criterion $B_{\bsgamma}(\bsg,f)$. We have
\begin{eqnarray*}
B_{\bsgamma}(\bsg,f) & = & \sum_{\emptyset \neq \uu \subseteq [s]} \gamma_\uu \prod_{i \in \uu} (b_i -a_i)  \sum_{\bsk_\uu \in \{1, \ldots, 2^m-1\}^{|u|}} \frac{1}{2^{\mu(\bsk_\uu)}} \frac{1}{2^m} \sum_{n=0}^{N-1} \mathrm{wal}_{\bsk_\uu}(\bsy_{n,\uu})\nonumber\\
& = &  \frac{1}{2^m} \sum_{n=0}^{2^m-1} \sum_{\emptyset \neq \uu \subseteq [s]} \gamma_\uu \prod_{i \in \uu} \left[(b_i -a_i)  \sum_{k = 1}^{2^m-1} \frac{1}{2^{\mu(k)}} \mathrm{wal}_{k}(y_{n,i})\right]\nonumber\\
& = &  \frac{1}{2^m} \sum_{n=0}^{2^m-1} \sum_{\emptyset \neq \uu \subseteq [s]} \gamma_\uu \prod_{i \in \uu} \left[(b_i -a_i)  \sum_{j=0}^{m-1} \frac{1}{2^{j+1}} \sum_{k = 2^j}^{2^{j+1}-1}  \mathrm{wal}_{k}(y_{n,i})\right]\nonumber
\end{eqnarray*}
since $\mu(k)=j+1$ for all $k \in \{2^j,\ldots,2^{j+1}-1\}$.

For $m$-bit numbers $x=\frac{\xi_1}{2}+\frac{\xi_2}{2^2}+\cdots +\frac{\xi_m}{2^m}$ with $\xi_l \in \{0,1\}$, we define  $$\phi(x):= \sum_{j=0}^{m-1} \frac{1}{2^j} \sum_{k = 2^j}^{2^{j+1}-1}  \mathrm{wal}_k(x).$$ Then we obtain   
\begin{equation}\label{effc1}
B_{\bsgamma}(\bsg,f)
= \frac{1}{2^m} \sum_{n=0}^{2^m-1} \sum_{\emptyset \neq \uu \subseteq [s]} \gamma_\uu \prod_{i\in \uu}\left[ \frac{b_i - a_i}{2}\, \phi(y_{n,i})\right].
\end{equation}

The $2^m$ values of $\phi(x)$ for $m$-bit numbers $x \in \{0,\tfrac{1}{2^m},\tfrac{2}{2^m},\ldots,\tfrac{2^m -1}{2^m}\}$ can be easily pre-computed: It is easily checked that 
\begin{equation*}
\frac{1}{2^j} \sum_{k=2^j}^{2^{j+1}-1} \wal_k(x) = \begin{cases} (-1)^{\xi_{j+1}} & \mbox{if } \xi_1 = \xi_2 = \ldots = \xi_j = 0, \\ 0 & \mbox{otherwise}, \end{cases}
\end{equation*}
and therefore we have 
\begin{equation}\label{effc2}
\phi(x)=\left\{ 
\begin{array}{ll}
 m & \mbox{ if $x=0$},\\
 i_0-2 & \mbox{ if $\xi_1=\ldots=\xi_{i_0-1}=0$ and $\xi_{i_0}=1$ where $i_0 \in \{1,2,\ldots,m\}$}. 
\end{array}\right.
\end{equation}

Now we consider two special types of weights:

\paragraph{Product weights.} For product weights of the form \eqref{prodW} we obtain from \eqref{effc1} $$B_{\bsgamma}(\bsg,f)=-1+\frac{1}{2^m}\sum_{n=0}^{2^m-1} \prod_{i=1}^s \left(1+\gamma_i \frac{b_i-a_i}{2}\, \phi(y_{n,i})\right)$$ and hence $B_{\bsgamma}(\bsg,f)$ can be computed in $O(sN)$ operations using \eqref{effc2}. 

\paragraph{POD weights.} For POD weights of the form \eqref{PODw} we can use a recursive computation of $B_{\bsgamma}$. We use an idea that was first proposed in \cite[Section~5]{KSS11}. From \eqref{effc1} we obtain \begin{eqnarray*}
\lefteqn{B_{\bsgamma}(\bsg,f)}\\
& = &  \frac{1}{2^m} \sum_{n=0}^{2^m-1}  \sum_{\emptyset \neq \uu \subseteq [s]} \Gamma_{|\uu|}  \prod_{i \in \uu} \left[\gamma_i \, \frac{b_i - a_i}{2} \, \phi(y_{n,i}) \right] \\
& = &  \frac{1}{2^m} \sum_{n=0}^{2^m-1} \sum_{\ell=1}^s \frac{\Gamma_{\ell}}{2^{\ell}}  \sum_{\uu \subseteq [s] \atop |\uu|=\ell}  \prod_{i \in \uu} \left[\gamma_j \, (b_i - a_i)\, \phi(y_{n,i})\right]\\
& = &  \frac{1}{N} \sum_{n=0}^{N-1} \sum_{\ell=1}^s \frac{\Gamma_{\ell}}{2^{\ell}} \Bigg[ \underbrace{\sum_{\uu \subseteq [s-1] \atop |\uu|=\ell}  \prod_{j \in \uu} (\gamma_j (b_j - a_j) \phi(y_{n,j}))}_{=:a_{s-1,\ell}(n)} + \gamma_s (b_s - a_s) \phi(y_{n,s}) \underbrace{\sum_{\uu \subseteq [s-1] \atop |\uu|=\ell-1}  \prod_{j \in \uu} (\gamma_j (b_j - a_j) \phi(y_{n,j}))}_{=: a_{s-1,\ell-1}(n)} \Bigg]\\
& = & B_{\bsgamma}((g_1,\ldots,g_{s-1}),f) + \frac{\gamma_s(b_s-a_s)}{N} \sum_{n=0}^{N-1}  \phi(y_{n,s}) \left(  \sum_{\ell=1}^s \frac{\Gamma_{\ell}}{2^{\ell}} a_{s-1,\ell-1}(n)\right).
\end{eqnarray*}
Hence, using the initial value $$B_{\bsgamma}((),f):=1$$ and the recursion
\begin{eqnarray*}
a_{s,0}(n) & := & 1 \\ 
a_{s,\ell}(n) & = & a_{s-1,\ell}(n)+ \gamma_s (b_s-a_s) \phi(y_{n,s}) a_{s-1,\ell-1}(n),
\end{eqnarray*}
we can compute the quality criterion $B_{\bsgamma}(\bsg,f)$ recursively.

\section{Strong polynomial tractability}\label{sec:SPT}

Tractability deals with the growth rate of the information complexity of a problem when the dimension  $s$ tends to infinity and the error demand $\varepsilon$ goes to zero. We explain this in more detail: Consider a generic sequence of integration problems $(I_s: \mathcal{H}_s \rightarrow \RR)_{s \ge 1}$, where $\mathcal{H}_s$ are normed spaces of $s$-variate functions and $I_s$ is the integral operator, with worst-case error ${\rm wce}$ depending on the integration nodes $\cP$ in dimension $s$. Then the {\it information complexity} of the problem is defined as $$N_{\bsgamma}(\varepsilon,s)=\min\{N \in \NN \ : \ \exists \cP \subseteq [0,1)^s, |\cP|=N \mbox{ such that } {\rm wce}(\cP) \le \varepsilon\}.$$ This is the minimal number of points that are required in order to approximate the integral in the worst case within an error of at most $\varepsilon$. There are several notions of tractability which classify the growth rate of the information complexity when $s \rightarrow \infty$ and $\varepsilon \rightarrow 0$. Here we only consider the notion of strong polynomial tractability. For general information about tractability studies we refer to the three volumes \cite{NW1,NW2,NW3} by Novak and Wo\'{z}niakowski.

\begin{definition}\rm
The integration problem is said to be {\it strongly polynomially tractable}, if there exist $C>0$ and $\tau>0$ such that 
\begin{equation}\label{def:SPT}
N_{\bsgamma}(\varepsilon,s) \le C \varepsilon^{-\tau} \ \ \ \ \mbox{ for all $s \in \NN$ and for all $\varepsilon \in (0,1]$,}
\end{equation}
i.e., if the information complexity is bounded polynomially in $\varepsilon^{-1}$ and uniformly in $s$. The infimum over all $\tau$ such that \eqref{def:SPT} holds is called the {\it exponent of strong polynomial tractability}. We denote it by $\tau^*$.
\end{definition}

The following lemma can be applied to the worst-case error bounds from this paper. 

\begin{lemma}\label{le:bd:info_comp}
Consider a generic sequence of integration problems $(I_s: \mathcal{H}_s \rightarrow \RR)_{s \ge 1}$, where $\mathcal{H}_s$ is a normed space of $s$-variate functions, with worst-case error ${\rm wce}$. Assume we have for every $s,m \in \NN$ a $2^m$-element point set $\cP_{s,m}$ in $[0,1]^s$ such that $${\rm wce}(\cP_{s,m}) \le \frac{C}{2^{m(1-\delta)}}$$ for some $\delta \in (0,1]$ and absolute constant $C>0$ (independent of $m$ and $s$, but which may depend on $\delta$ and on maybe other parameters). Then we have strong polynomial tractability with exponent $\tau^*$ at most $1/(1-\delta)$. 
\end{lemma}

\begin{proof}
Let $\varepsilon > 0$. Choose $$m=\left\lceil \frac{\ld(C \varepsilon^{-1})}{1-\delta}\right\rceil .$$ Then $${\rm wce}(\cP) \le \frac{C}{2^{m(1-\delta)}} \le \varepsilon.$$ This implies that $$N_{\bsgamma}(\varepsilon,s) \le 2^{\left\lceil \frac{\ld(C \varepsilon^{-1})}{1-\delta}\right\rceil} \le 2 (C \varepsilon^{-1})^{1/(1-\delta)}.$$ Hence we achieve strong polynomial tractability with exponent $\tau^*$ at most $1/(1-\delta)$. 
\end{proof}

In the following we will look at several constructions of digital nets and check how the weights have to be chosen in order to satisfy the condition in Lemma~\ref{le:bd:info_comp} and to achieve strong polynomial tractability.

\subsection{Niederreiter and Sobol' digital sequences}

A number of explicit constructions of digital nets are known. For a survey we refer to \cite[Chapter~8]{DP10}. In the following we study the dependence of the worst-case error on the dimension for two well established explicit constructions of digital nets for general weights.

There are explicit constructions of digital $((t_\uu)_\uu, m, s)$-nets over $\mathbb{F}_2$ due to Sobol'~\cite{sob1967} and Niederreiter~\cite{nie88}. For nets obtained from a Niederreiter sequence the quality parameter $t_\uu$ satisfies
\begin{equation*}
t_\uu \le \sum_{i \in \uu} \left( \ld (i) + \ld \ld (i+2) + 2 \right),
\end{equation*}
where $\ld$ denotes the logarithm in base $2$, see \cite[Lemma~2]{W02}. Hence
\begin{equation*}
2^{t_\uu} \le \prod_{i \in \uu} (4 \, i \, \ld (i + 2) ).
\end{equation*}

Using this bound, we obtain from Theorem~\ref{thm2}
\begin{equation}\label{corthm2:tract}
{\rm wce}(\cP; p,s,\bsgamma, \varphi)\le \frac{3}{2^m}\,\left(\sum_{\emptyset \neq \uu \subseteq [s]}  \left(\gamma_\uu \, m^{|\uu|}\, \prod_{i \in \uu} \left[8\, i\, \ld(i+2)\, (b_i - a_i)\right]\right)^q \right)^{1/q},
\end{equation}
where $1 \le p, q \le \infty$ and $\frac{1}{p}+\frac{1}{q}=1$, with the obvious modifications for $q = \infty$.

\begin{theorem}\label{thm_dignet_prod}
Let $\cP$ be a digital net obtained from a Niederreiter sequence. Assume there exists a $\delta \in (0,1)$ such that 
\begin{equation}\label{cond:weight}
C_{{\rm Nied}}(\delta,\bsgamma,q) :=\sup_{s \in \NN} \left( \sum_{\emptyset \neq \uu \subseteq [s]}  \left(\gamma_\uu \,  \prod_{i \in \uu} \left[\frac{8 \, |\uu|\, i}{{\rm e}\, \delta\, \ln 2} \, \ld(i+2)\, (b_i - a_i)\right]\right)^q \right)^{1/q} < \infty,
\end{equation}
with the obvious modifications for $q = \infty$. Then we have 
\begin{equation}\label{errbd:tract}
{\rm wce}(\cP; p,s,\bsgamma, \varphi)\le \frac{3\, C_{{\rm Nied}}(\delta,\bsgamma,q)}{2^{m(1-\delta)}},
\end{equation}
where $1 \le p, q \le \infty$ and $\frac{1}{p}+\frac{1}{q}=1$.

For strictly positive weights $\gamma_{\uu}$, if
\begin{equation}\label{cond:posweight}
S_{{\rm Nied}}(\bsgamma):=\sum_{i=1}^{\infty} [i\, \ld(i+2)\, (b_i-a_i)]\, \max_{\vv \subseteq [i-1]} \frac{\gamma_{\vv \cup \{i\}}}{\gamma_{\vv}} < \infty,
\end{equation}
where in the above $\gamma_\emptyset:=1$, then for any 
\begin{equation*}%\label{def:delta}
\delta \in \left(0,\min\left\{1,\frac{16 \, S_{{\rm Nied}}(\bsgamma)}{\ln 2}\right\}\right)
\end{equation*}
we have 
\begin{equation*}%\label{errbd:tract:prod}
{\rm wce}(\cP; p,s,\bsgamma, \varphi)\le \frac{C_{{\rm Nied,prod}}(\delta,\bsgamma)}{2^{m(1-\delta)}}, 
\end{equation*}
where $$C_{{\rm Nied,prod}}(\delta,\bsgamma) = 3 \left( \frac{16\, S_{{\rm Nied}}(\bsgamma)}{\delta \, \ln 2} \right)^b,$$ and the integer $b$ is such that 
\begin{equation*}%\label{def:b}
\sum_{i=b+1}^\infty [i\, \ld(i+2)\, (b_i - a_i)] \,\max_{\vv \subseteq [i-1]} \frac{\gamma_{\vv \cup \{i\}}}{\gamma_{\vv}} < \frac{\delta\, \ln 2}{16}.
\end{equation*}
\end{theorem}

\begin{proof}
For $\delta >0$, the real function $f(x)=\frac{x^{|\uu|}}{2^{\delta x}}$ attains its maximum in $x=|\uu|/(\delta \ln 2)$ and hence 
\begin{equation}\label{absch:mhu}
m^{|\uu|} \le 2^{\delta m} f\left(\frac{|\uu|}{\delta \ln 2}\right) = 2^{\delta m} \left(\frac{|\uu|}{{\rm e}\, \delta \, \ln 2 } \right)^{|\uu|}.
\end{equation}
Inserting this estimate into \eqref{corthm2:tract} we obtain
\begin{eqnarray*}
{\rm wce}(\cP; p,s,\bsgamma, \varphi) & \le & \frac{3}{2^{m(1-\delta)}}\,\left(\sum_{\emptyset \neq \uu \subseteq [s]}  \left(\gamma_\uu \,  \prod_{i \in \uu} \left[\frac{8\, |\uu|\, i}{{\rm e} \, \delta\, \ln 2} \, \ld(i+2)\, (b_i - a_i)\right]\right)^q \right)^{1/q}\\
& \le & \frac{3 \, C_{{\rm Nied}}(\delta,\bsgamma,q)}{2^{m(1-\delta)}}.
\end{eqnarray*}
This proves \eqref{errbd:tract}.

Now we prove the assertion for strictly positive weights. Assume that all weights $\gamma_\uu$ are strictly positive. Using Jensen's inequality, \eqref{corthm2:tract} can be re-written as
\begin{equation*}
{\rm wce}(\cP; p,s,\bsgamma, \varphi)  \le  \frac{3}{2^m}\, \sum_{\emptyset \neq \uu \subseteq [s]}  \nu_{\uu} m^{|\uu|},  
\end{equation*}
where $$\nu_{\uu}:=\gamma_{\uu} \prod_{i \in \uu} \left[8\, i\, \ld(i+2)\, (b_i - a_i)\right].$$
Now we use \cite[Lemma~4, Eq. (15)]{EKNO} which states that $$\sum_{\emptyset \neq \uu \subseteq [s]}  \nu_{\uu} m^{|\uu|} \le \prod_{i=1}^{s} (1+\widetilde{\nu}_i m),\quad \mbox{ where }\, \widetilde{\nu}_i:= \max_{\vv \subseteq [i-1]} \frac{\nu_{\vv \cup \{i\}}}{\nu_{\vv}}.$$ This implies 
\begin{equation}
{\rm wce}(\cP; p,s,\bsgamma, \varphi) \le \frac{3}{2^m}\, \prod_{i=1}^{s} (1+\widetilde{\nu}_i m).
\end{equation}
For any $\delta \in (0,\min\{1,\frac{2}{\ln 2} \sum_{i=1}^{\infty} \widetilde{\nu}_i\})$, where $a:=\frac{2}{\delta \ln 2} \sum_{i=1}^{\infty}\widetilde{\nu}_i $ and $b$ such that $\sum_{i>b} \widetilde{\nu}_i < \frac{\delta \ln 2}{2}$, we obtain from \cite[Lemma~4.4]{KSS11} that 
$$\prod_{i=1}^{s} (1+\widetilde{\nu}_i m) \le a^b 2^{\delta m}.$$
From this we obtain 
\begin{equation*}
{\rm wce}(\cP; p,s,\bsgamma, \varphi) \le \frac{3\, a^b}{2^{m(1-\delta)}}.
\end{equation*}

It remains to re-write the condition on the weights. Obviously,
$$\widetilde{\nu}_i= [8 \, i \, \ld(i+2) (b_i-a_i)] \max_{\vv \subseteq [i-1]} \frac{\gamma_{\vv \cup \{i\}}}{\gamma_{\vv}}.$$ Hence we require the weight condition 
$$S_{{\rm Nied}}(\bsgamma):=\sum_{i=1}^\infty [i \, \ld(i+2) \, (b_i-a_i)]\, \max_{\vv \subseteq [i-1]} \frac{\gamma_{\vv \cup \{i\}}}{\gamma_{\vv}} < \infty.$$ In this case we have for any $$\delta \in \left(0,\min\left\{1, \frac{16\, S_{{\rm Nied}}(\bsgamma)}{\ln 2} \right\} \right),$$ that $a$ is given by  $$ a=\frac{16\, S_{{\rm Nied}}(\bsgamma)}{\delta \, \ln 2},$$ and $b$ is such that $$\sum_{i>b} [i\, \ld(i+2)\, (b_i-a_i)]\,\max_{\vv \subseteq [i-1]} \frac{\gamma_{\vv \cup \{i\}}}{\gamma_{\vv}} < \frac{\delta \, \ln 2}{16}.$$ This proves the desired result.
\end{proof}

\begin{remark}\rm\label{condweightprodPOD}
For product weights $\gamma_\uu =\prod_{i \in \uu} \gamma_i$ condition \eqref{cond:posweight} is equivalent to $$\sum_{i=1}^{\infty} [\gamma_i \, i \, \ld(i+2)\, (b_i-a_i)]< \infty.$$
For POD weights $\gamma_\uu= \Gamma_{|\uu|}\prod_{i \in \uu} \gamma_i $ we have $$\max_{\vv \subseteq [i-1]} \frac{\gamma_{\vv \cup \{i\}}}{\gamma_{\vv}}=\gamma_i \, \max_{v \in [i-1]} \frac{\Gamma_{v+1}}{\Gamma_{v}}$$ and hence condition \eqref{cond:posweight} is equivalent to $$\sum_{i=1}^{\infty} [\gamma_i \, i \, \ld(i+2)\, (b_i-a_i)] \,\max_{v \in [i-1]} \frac{\Gamma_{v+1}}{\Gamma_{v}} < \infty.$$ For example, if $\Gamma_t=(t!)^\lambda$ with some $\lambda>0$, then $\max_{v \in [i-1]} \frac{\Gamma_{v+1}}{\Gamma_{v}}=i^\lambda$ and then condition \eqref{cond:posweight} is equivalent to $$\sum_{i=1}^{\infty} [\gamma_i \, i^{1+\lambda} \, \ld(i+2)\, (b_i-a_i)] < \infty.$$
\end{remark}

For nets obtained from a Sobol' sequence, the quality parameter $t_\uu$ satisfies (see \cite[Eq.~15]{W02b})
\begin{equation}\label{tuSob}
t_\uu \le \sum_{i \in \uu} \left( \ld(i) + \ld \ld (i+1) + \ld \ld \ld (i+3) + c \right),
\end{equation}
for some constant $c >0$ independent of $i, \uu$ and $s$. Hence, in the same way as above, we obtain the following theorem.

\begin{theorem}
Let $\cP$ be a digital net obtained from a Sobol' sequence. Assume there exists a $\delta \in (0,1)$ such that 
\begin{equation*}
C_{{\rm Sob}}(\delta,\bsgamma,q) :=\sup_{s \in \NN}  \left( \sum_{\emptyset \neq \uu \subseteq [s]}  \left(\gamma_\uu \,  \prod_{i \in \uu} \left[\frac{2^{c+1} \, |\uu|\, i }{{\rm e}\, \delta\, \ln 2} \, \ld(i+1)\, \ld\ld(i+3)\, (b_i - a_i)\right]\right)^q \right)^{1/q} < \infty,
\end{equation*}
where $c$ is from \eqref{tuSob}, with the obvious modifications for $q = \infty$. Then we have 
\begin{equation*}
{\rm wce}(\cP; p,s,\bsgamma, \varphi)\le \frac{3\, C_{{\rm Sob}}(\delta,\bsgamma,q)}{2^{m(1-\delta)}},
\end{equation*}
where $1 \le p, q \le \infty$ and $\frac{1}{p}+\frac{1}{q}=1$.

For strictly positive weights, if 
\begin{equation}\label{weightcondSob}
S_{{\rm Sob}}(\bsgamma):= \sum_{i = 1}^\infty \left[i \,  \ld (i+1)\, \ld\ld(i+3) \, (b_i - a_i)\right] \, \max_{\vv \subseteq [i-1]} \frac{\gamma_{\vv \cup \{i\}}}{\gamma_{\vv}}  < \infty,
\end{equation}
where in the above $\gamma_\emptyset:=1$, then for any 
\begin{equation*}
\delta \in \left(0,\min\left\{1,\frac{2^{c+2} S_{{\rm Sob}}(\bsgamma)}{\ln 2} \right\}\right)
\end{equation*}
we have 
\begin{equation*}
{\rm wce}(\cP; p,s,\bsgamma, \varphi)\le \frac{C_{{\rm Sob, prod}}(\delta,\bsgamma)}{2^{m(1-\delta)}}, 
\end{equation*}
where $$C_{{\rm Sob,prod}}(\delta,\bsgamma) = 3 \left( \frac{2^{c+2}S_{{\rm Sob}}(\bsgamma)}{\delta \, \ln 2} \right)^b ,$$ and the integer $b$ is such that 
\begin{equation*}
\sum_{i=b+1}^\infty [i\, \ld(i+1)\, \ld\ld(i+3)\, (b_i - a_i)]\, \max_{\vv \subseteq [i-1]} \frac{\gamma_{\vv \cup \{i\}}}{\gamma_{\vv}}  < \frac{\delta\, \ln 2}{2^{c+2}}.
\end{equation*}
\end{theorem}

Like in Remark~\ref{condweightprodPOD} one may easily re-write the weight condition \eqref{weightcondSob} for product- and POD weights.

\subsection{Polynomial lattices}

We also consider polynomial lattices. Again we restrict ourselves to the case $p=\infty$ and hence $q=1$.

\begin{theorem}\label{thm_spt_prod}
Let  $f \in \FF_2[x]$ be irreducible with $\deg(f) =m$. Suppose that $\bsg^* \in G_m^s$ is constructed according to Algorithm~\ref{alg_weight}. 

Assume there exists a $\delta \in (0,1)$ such that 
\begin{equation*}
C_{{\rm Poly}}(\delta,\bsgamma) :=\sup_{s \in \NN} \sum_{\emptyset \neq \uu \subseteq [s]}   \gamma_\uu \, \prod_{i \in \uu} \left[\max\left(2,\frac{|\uu|}{{\rm e}\, \delta \, 2 \ln 2}\right) \,(b_i - a_i)\right] < \infty.
\end{equation*}
Then we have 
\begin{equation*}
{\rm wce}(\cP(\bsg^*,f); \infty,d,\bsgamma, \varphi) \le  \frac{2\, C_{{\rm Poly}}(\delta,\bsgamma)}{2^{m(1-\delta)}}.
\end{equation*}

For strictly positive weights $\gamma_\uu$, if 
\begin{equation}\label{condweightposPoly}
S(\bsgamma):=\sum_{i = 1}^\infty (b_i - a_i) \, \max_{\vv \subseteq [i-1]} \frac{\gamma_{\vv \cup \{i\}}}{\gamma_{\vv}} < \infty,
\end{equation}
where $\gamma_\emptyset:=1$, then for any 
\begin{equation*}
\delta \in \left(0,\min\left\{1, \frac{5\, S(\bsgamma)}{\ln 2}\right\}\right)
\end{equation*}
we have  
$${\rm wce}(\cP(\bsg^*,f); \infty,d,\bsgamma, \varphi)  \le  \frac{C_{{\rm Poly,prod}}(\delta,\bsgamma)}{2^{m(1-\delta)}},$$ where
$$C_{{\rm Poly,prod}}(\delta,\bsgamma)= \left[\frac{5\, S(\bsgamma)}{\delta \, \ln 2} \right]^b,$$ and the integer $b$ is such that 
\begin{equation*}
\sum_{i=b+1}^\infty (b_i - a_i) \, \max_{\vv \subseteq [i-1]} \frac{\gamma_{\vv \cup \{i\}}}{\gamma_{\vv}}< \frac{\delta\, \ln 2}{5}.
\end{equation*}
\end{theorem}

\begin{proof}
Suppose that $\bsg^*\in G_m^s$ is constructed according to Algorithm~\ref{alg_weight}. From Corollary~\ref{co1} and \eqref{absch:mhu} we obtain  
\begin{eqnarray*}
{\rm wce}(\cP(\bsg^*,f); \infty,d,\bsgamma, \varphi) & \le & \frac{1}{2^m} \sum_{\emptyset \neq \uu \subseteq [s]}  \gamma_\uu \, \left(2^{|\uu|} + \left(\frac{m}{2}\right)^{|\uu|} \right)\, \prod_{i \in \uu} (b_i - a_i)\\
& \le & \frac{1+2^{\delta m}}{2^m} \sum_{\emptyset \neq \uu \subseteq [s]}  \gamma_\uu \, \prod_{i \in \uu} \left[\max\left(2,\frac{|\uu|}{{\rm e}\, \delta \, 2 \ln 2}\right) \,(b_i - a_i)\right]\\
& \le & \frac{2\, C_{{\rm Poly}}(\delta,\bsgamma)}{2^{m(1-\delta)}}.  
\end{eqnarray*}

For strictly positive weights, the upper bound in Corollary~\ref{co1} can be estimated by
\begin{eqnarray*}
{\rm wce}(\cP(\bsg^*,f); \infty,d,\bsgamma, \varphi)  & \le & \frac{1}{2^m}\sum_{\emptyset \not=\uu \subseteq [s]} \nu_{\uu} m^{|\uu|}, 
\end{eqnarray*}
where $$\nu_\uu:= \gamma_\uu \, \prod_{i \in \uu} \left[\frac{5}{2} \, (b_i - a_i)\right].$$ 
Then for arbitrary $\delta$ as in the statement of the theorem, the proof follows in the same way as the proof of Theorem~\ref{thm_dignet_prod}.
\end{proof}

Like in Remark~\ref{condweightprodPOD} we can re-write weight condition \eqref{condweightposPoly} for product- and for POD weights.

\begin{remark}\rm
For product weights $\gamma_\uu =\prod_{i \in \uu} \gamma_i$ condition \eqref{condweightposPoly} is equivalent to $$\sum_{i=1}^{\infty} [\gamma_i \, (b_i-a_i)]< \infty.$$
For POD weights condition \eqref{condweightposPoly} is equivalent to $$\sum_{i=1}^{\infty} [\gamma_i \, (b_i-a_i)] \,\max_{v \in [i-1]} \frac{\Gamma_{v+1}}{\Gamma_{v}} < \infty.$$ In particular, for $\Gamma_t=(t!)^\lambda$ with some $\lambda>0$,  condition \eqref{condweightposPoly} is equivalent to $$\sum_{i=1}^{\infty} [\gamma_i \, i^{\lambda} \,  (b_i-a_i)] < \infty.$$
\end{remark}

\vspace{0.5cm}
\noindent{\bf Author's Addresses:}

\noindent Josef Dick, School of Mathematics and Statistics, The University of New South Wales, Sydney, NSW 2052, Australia.  Email: josef.dick@unsw.edu.au \\

\noindent Friedrich Pillichshammer, Institut f\"{u}r Analysis, Universit\"{a}t Linz, Altenbergerstra{\ss}e 69, A-4040 Linz, Austria. Email: friedrich.pillichshammer@jku.at

\end{document}